\documentclass[11pt,letterpaper]{article}
\usepackage[titletoc]{appendix}
\usepackage{amsmath,amsfonts,fancyhdr,url, amsthm, centernot}
\usepackage{lscape}
\usepackage{mathtools}
\usepackage[pdftex]{graphicx}
\graphicspath{{}}
\usepackage{extramarks}
\usepackage{subfiles}
\usepackage{geometry}
\usepackage{afterpage}
\usepackage{lscape}
\usepackage{soul}
\usepackage{graphicx}
\usepackage{float}
 \usepackage{amssymb}
\usepackage{caption}
\usepackage{color}
\usepackage{algorithm}
\usepackage{longtable}
\usepackage[noend]{algpseudocode}

\usepackage[usenames,dvipsnames,svgnames,table]{xcolor}
\usepackage{wrapfig}
\usepackage{sectsty}
\usepackage{subfig}
\usepackage{hhline}
\usepackage{algpseudocode}
\usepackage{enumerate}
\usepackage{tikz}
\usetikzlibrary{decorations.pathreplacing,calc}

\fancypagestyle{firstpage}{%
  \lhead{}
  \rhead{}
  \cfoot{}
}

\newcommand{\roll}{\mathrm{roll}}
\newcommand{\step}{\mathrm{step}}
\newcommand{\lrp}[1]{\left(#1\right)}
\newcommand{\floor}[1]{\left \lfloor#1 \right \rfloor}
\newcommand{\ceil}[1]{\left \lceil#1\right \rceil}
\newcommand{\defeq}{\vcentcolon=}

\newcommand{\ds}{\displaystyle}
\newcommand{\dsum}{\ds\sum}

\newcommand{\Center}{Center }
\newcommand{\xsf}{\mathsf{x}}
\newcommand{\fsf}{\mathsf{f}}
\newcommand{\alphasf}{\mathsf{\alpha}}
\newcommand{\betasf}{\mathsf{\beta}}
\newcommand{\ysf}{\mathsf{y}}
\newcommand{\vsf}{\mathsf{v}}
\newcommand{\bargeNumUnloads}{\underbar{bul}^{\max}}

\sectionfont{\fontsize{18}{15}\selectfont}

\newcommand{\Nonlinear}{\framebox{Bi-linear} }%

\newcommand{\rat}{\text{rat}}
\newcommand{\INB}{y^{\text{in}}}
\newcommand{\INV}{v^{\text{end}}}
\newcommand{\SPL}{v^{\text{mid}}}
\newcommand{\FEED}{y^{\text{out}}}
\newcommand{\vd}{v^{\text{prod}}}

\newcommand{\relaxation}{\mathrm{relax}}

\newcommand{\SPEC}{f}
\newcommand{\SPECd}{\SPEC^{\text{prod}}}
\newcommand{\thalf}{t}
\renewcommand{\j}{{\kappa}}
\newcommand{\misInbPenalty}{\underbar{c}^{\mathrm{inb}}}%
\newcommand{\KQT}{}%

\newcommand{\misPenalty}{\underbar{c}^{\mathrm{dmd}}}%
\newcommand{\minTankFeedPct}{\underbar{tkp}^{\min}}%
\newcommand{\maxUnloadsPerDay}{\underbar{ul}^{\max}}%

\newcommand{\SVOLINV}{\vsf^{\fsf,\text{end}}}
\newcommand{\SVOLINVdelta}{\vsf^{\Delta\fsf,\text{end}}}
\newcommand{\SVOLSPL}{\vsf^{\fsf,\text{mid}}}
\newcommand{\SVOLSPLdelta}{\vsf^{\Delta\fsf,\text{mid}}}
\newcommand{\SVOLFEED}{\ysf^{\fsf,\text{out}}}
\newcommand{\SVOLFEEDdelta}{\ysf^{\Delta\fsf,\text{out}}}
\newcommand{\AVOLINV}{\vsf^{\alpha,\text{end}}}
\newcommand{\AVOLSPL}{\vsf^{\alpha,\text{mid}}}
\newcommand{\AVOLFEED}{\ysf^{\alpha,\text{out}}}

\newcommand{\spec}{\underbar{\SPEC}}

\newcommand{\demand}{\underbar d}
\newcommand{\MIS}{d^{\text{unmet}}}%
\newcommand{\invInit}{\underbar{v}}

\DeclareMathOperator    \proj           {proj}

\newcommand{\bb}{\mathbb}
\newcommand{\old}[1]{{}}

\newcommand{\R}{\bb R}

\newcommand{\specRat}{\underbar{r}}
\newcommand{\invInb}{\underbar{v}}
\newcommand{\MISINB}{v^{\text{unused}}}

\newcommand{\MINUNLTIME}{t^{\text{first}}}%
\newcommand{\MAXUNLTIME}{t^{\text{last}}}%
\newcommand{\maxUnloadTimeGap}{\underbar{ult}^{\max}}%
\newcommand{\minDailyUnloadPct}{\underbar{ulp}^{\min}}%

\newcommand{\low}{\min}
\newcommand{\upp}{\max}

\newcommand{\tfix}{fixed}
\newcommand{\tuse}{active}
\newcommand{\trelax}{relaxed}
\newcommand{\tomit}{omitted}

\newtheorem{theorem}{Theorem}

\newtheorem{proposition}[theorem]{Proposition}

\title{An Approximate Method for the Optimization of Long-Horizon Tank Blending and Scheduling Operations}%
\author{Benjamin Beach\footnote{Grado Department of Industrial and Systems Engineering, Virginia Tech. Corresponding author, e-mail: \texttt{bben6@vt.edu}},  
Robert Hildebrand\footnote{Grado Department of Industrial and Systems Engineering, Virginia Tech},
Kimberly Ellis\footnotemark[2]%
, and %
Baptiste Lebreton\footnote{Operations Research Team Leader, Eastman Chemical Company, Kingsport, TN, email: blebreton@eastman.com}
}

\begin{document}

\maketitle

\begin{abstract}
    We address a challenging tank blending and scheduling problem regarding operations for a chemical plant. We model the problem as a nonconvex MIQCP, then approximate this model as a MILP using a discretization-based approach. We combine a rolling horizon approach with the discretization of individual chemical property specifications to deal with long scheduling horizons, time-varying quality specifications, and multiple suppliers with discrete arrival times. This approach has been evaluated using industry-representative data sets from the specialty chemical industry.
    We demonstrate that the proposed approach supports fast planning cycles.
\end{abstract}

\thispagestyle{firstpage}

\section{Introduction}
\label{sec:introduction}

For speciality chemical manufacturing, companies purchase and blend supply streams from a variety of suppliers. Because the supply streams are derivatives of refinery operations, the chemical composition (specifications) of the streams may vary significantly by supplier as well as by season. Supply streams may have both desirable properties and less desirable properties that prevent their exclusive use during the production process.  Furthermore, some streams might have limited available quantities or intermittent availability.  Thus, supply streams of differing composition are often combined into intermediate storage tanks.  The resulting compounds are then blended to meet demand for production feed (see Figure~\ref{fig_probdesc}).  

For our motivating application, supply streams typically arrive in barges from suppliers and may remain at the inbound docks for several days before unloading into one or multiple storage  tanks.  The time windows for the arrival and departure of barges are committed several months in advance.  As the supply barges are unloaded into the tanks, the volume and chemical composition of the compounds in the storage tanks are changed. The compounds from one or multiple storage tanks are then blended to create the desired quantity and specification of production feed.  The production feed is planned in product campaigns (runs), typically spanning 3-14 days per product.  Approximately 3 to 12 months of product campaigns are planned in advance due to the wide variety of spec requirements over time for very different non-concurrent production runs, the non-periodic nature of the supply acquisition process, and the limited capacity of the storage tanks. %

Given the volume and characteristics of the compounds on the inbound barges, the time windows for the availability of barges, and the desired production feed (both quantity and specifications), the company must determine an assignment of supply barges to tanks, an unloading plan from barges to tanks, and a blending plan from tanks to production feed. Although similar problems are often inherently multi-objective, the most important objective by far is to satisfy production requirements and to unload supply. Hence, the objective is to determine an operational schedule for the given barge arrivals and production plan. Rapid solutions are required for long scheduling horizons to support effective planning decisions.  Thus, we propose an efficient optimization-based approach to provide solutions for this assignment, mixing, and blending problem within ten minutes for a planning horizon of up to a year.

\begin{figure}[H]
\begin{minipage}{\textwidth}
	\centering
\begin{tikzpicture}[scale = 1.0]

\draw[thick,densely dashed, gray] (-2,-1) rectangle (10.8,6.0);

\node (A)at (0,0) {\includegraphics[scale=.08]{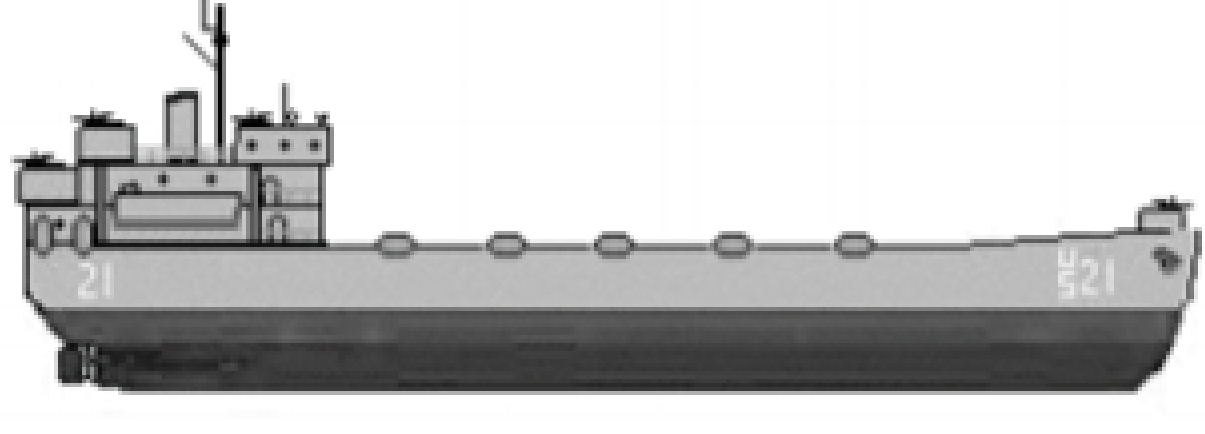}};
\node[fill=white,inner sep=1pt,rounded corners] at (0,0.4){ Barge $|S|$};
        
\node (B)at (0,2.5) {\includegraphics[scale=.08]{boat-image.png}};

\node[fill=white,inner sep=1pt,rounded corners] at (0,2.9){ Barge $s$};

\node[fill=white,inner sep=1pt,rounded corners] at (0,2.2){ $\invInb_s$, $\spec_{s,q}$};%

\node (C)at (0,5) {\includegraphics[scale=.08]{boat-image.png}};

\node[fill=white,inner sep=1pt,rounded corners] at (0,5.4){ Barge $1$};

\draw[dotted,line width=2.5pt,gray] (0,4) -- ++(0,-.45);
        
\draw[dotted,line width=2.5pt,gray] (4.5,4.1) -- ++(0,-.45);
        
\draw[dotted,line width=2.5pt,gray] (0,1.5) -- ++(0,-.45);
        
\draw[dotted,line width=2.5pt,gray] (4.5,1.5) -- ++(0,-.45);

\begin{scope}[shift={(4.5,4.5)}]
\draw[left color=black!40,draw=black!75] (1,0) arc(0:-180:1 and 0.1) -- (-1,.8)arc(-180:0:1 and 0.1) -- (1,0) node[midway](TK1){};
\draw[fill=black!5,draw=black!75] (0,0.8) ellipse (1 and 0.1) node[midway,above=1cm]{ Tank 1} ;
\end{scope}

\begin{scope}[shift={(4.5,2)}]
\draw[left color=blue!20,draw=black!75] (1,0) arc(0:-180:1 and 0.1) -- (-1,.8)arc(-180:0:1 and 0.1) -- (1,0) node[midway](TK2){};
\draw[fill=black!5,draw=black!75] (0,0.8) ellipse (1 and 0.1)node[midway,above=1cm]{ Tank $k$} node[midway,above, text width=2cm,align=center,black]{$v_{k,t}$,  %
$\SPEC_{k,q,t}$ };%
\end{scope}

\begin{scope}[shift={(4.5,-0.5)}]
\draw[left color=black!40,draw=black!75] (1,0) arc(0:-180:1 and 0.1) -- (-1,.8)arc(-180:0:1 and 0.1) -- (1,0) node[midway](TK3){};
\draw[fill=black!5,draw=black!75] (0,0.8) ellipse (1 and 0.1) node[midway,above=.85cm]{ Tank $|K|$};
\end{scope}

\node[fill=white] at (8.6,4.2) {Production Feed};
\draw[-latex,thick,gray] (C.-5) -- ++(1.65,-4.8);
\draw[-latex,thick,gray] (B.-5) -- ++(1.65,-2.5);
\draw[-latex,thick,gray] (A.-5) -- ++(1.65,0);

\draw[-latex,thick,gray] (C.-5) -- ++(1.65,0);
\draw[-latex, thick,gray] (B.-5) -- ++(1.65,2.5) ;
\draw[-latex, thick,gray] (A.-5) -- ++(1.65,4.8);

\draw[-latex,very thick,blue] (C.-5) -- ++(1.65,-2.5); %
\draw[-latex,very thick,blue] (B.-5) -- ++(1.65,0)node[midway,above,blue,fill=white, rounded corners,inner sep=0, fill opacity=.9]{ $\INB_{s,k,t}$};
\draw[-latex,very thick,blue] (A.-5) -- ++(1.65,2.5);

\node[draw,circle,fill=yellow!40,minimum height=1cm,minimum width=1.5cm] (GB1)at ($(B.east)+(6.5,-0.1)$){$\vd_t, \SPECd_{q,t} $};

\node[draw=white,fill=white,minimum height=1cm,minimum width=1.7cm] (GB2)at ($(B.east)+(10.2,-0.1)$){};

\node[minimum height=1cm,minimum width=1.5cm] at ($(B.east)+(10.2,0.1)$){\includegraphics[scale=0.08]{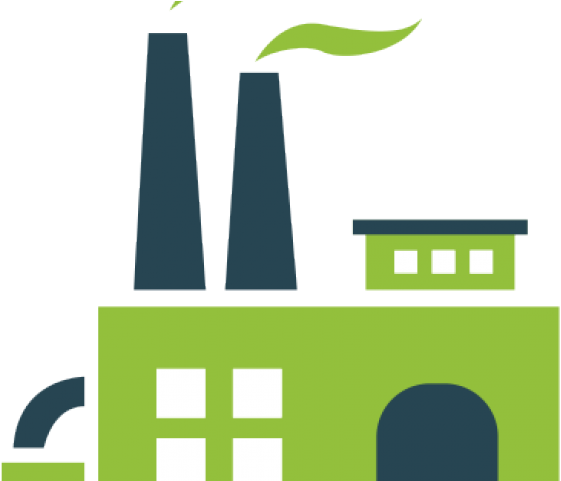}};
\node[below=.5cm] at (GB2) { Production};

\draw[-latex,thick] (TK1.center) -- (GB1.160);
\draw[-latex,thick] (TK2.center) -- (GB1.west) node[near start,right=0.4cm, above]{ $\FEED_{k,t}$};%
\draw[-latex,thick] (TK3.center) -- (GB1.200) ;

\draw[-latex,thick] (GB1.east) -- (GB2.west);%

\end{tikzpicture}      %
  \caption[We schedule the the front-end portion of the problem, before the production stage (as indicated by the large box). This includes the scheduling of when and where to unload each barge and the determination of flow plans from tanks to production feed.]{We schedule the the front-end portion of the problem, before the production stage (as indicated by the large box). This includes the scheduling of when and where to unload each barge and the determination of flow plans from tanks to production feed.} %
  \label{fig_probdesc}
  \end{minipage}
\end{figure}

Within the context of supply chain management, this problem encompasses short-term decisions as highlighted in the Supply Chain Planning Matrix (Figure~\ref{fig_planningmatrix}, adapted from \cite{Rohde2000}).

\begin{figure}[H]

\begin{minipage}{\textwidth}
	\centering
    \includegraphics[scale=0.7]{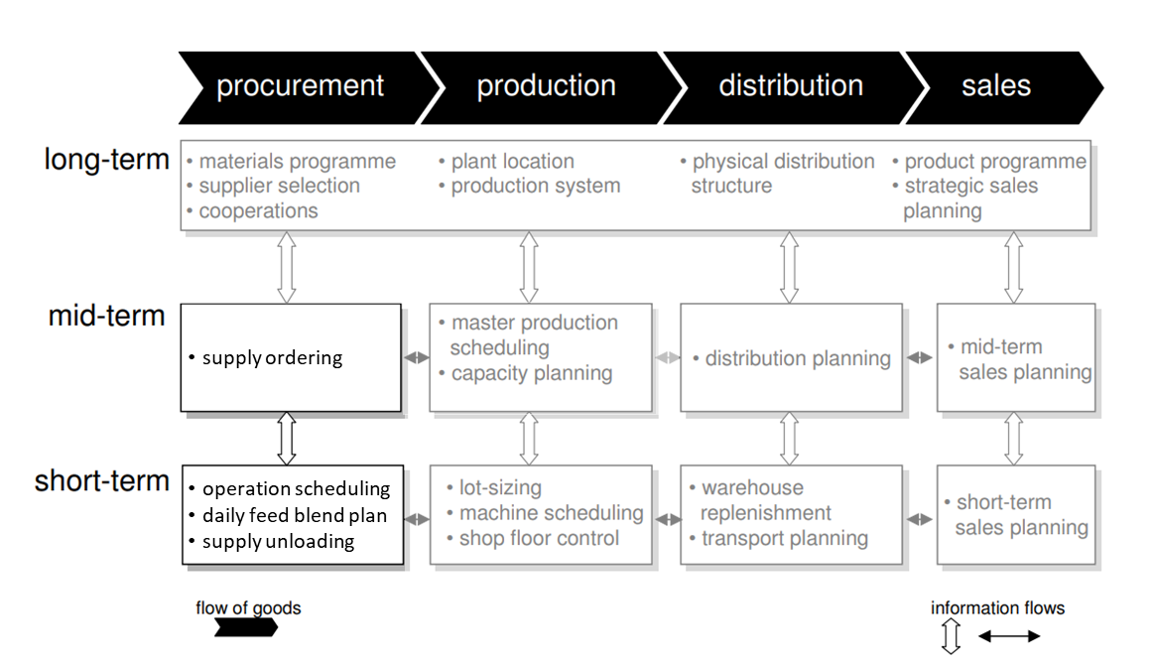} 
  \caption{Context of this work in terms of supply chain planning (adapted from \cite{Rohde2000}).}
  \label{fig_planningmatrix}
  \end{minipage}
\end{figure}

In this work, we employ a discrete-time mixed-integer linear program (MILP) approximate formulation of the core nonconvex mixed-integer quadratically-constrained program (MIQCP), formulated via direct tracking of specs, using single-day scheduling periods with scheduling horizons of 90 to 368 days. The `binary' version of this discretization combines aspects of multi-parametric disaggregation (MDT) \cite{Teles2011} with normalized MDT (NMDT, \cite{Castro2015c}), while avoiding the explicit tracking of spec qualities. The demand spec and spec ratio requirements are tightened to help ensure that the approximate representation of the specs does not lead to violations of the demand composition requirements. To satisfy the long-horizon scheduling goals within a reasonable computation time, we use a rolling-horizon solution approach.
After optimization of the model approximation, the scheduling solution is simulated to obtain the correct tank inventory and demand feed specs.

In the following sections, we provide a formal problem description, highlight the related research, formulate a mixed integer non-linear representation of the problem, propose approaches for solving the model, and compare the approaches with computational studies based on industry-representative data sets.

\section{Problem Description}
For a given a set $S$ of supply nodes (barges), assume each supply node $s$ contains a given volume $\invInb_s$ of raw material, with a level for chemical property $q$ of $\spec_{s,q}$, where $q \in Q$. The levels of these chemical properties are often specified as a percentage of the volume for a material component of interest. The percentages may not necessarily sum to 100\% of the material, because some components may overlap and others may not be of interest. Each supply node $s$ is available during a set $T_s$ of consecutive time periods, from $t^{\min}_s$ to $t^{\max}_s$. With the availability of all supply nodes known, the set of supply nodes available each time period $t$ is denoted as $S_t$. The number of barges that are unloaded each day is limited to $\maxUnloadsPerDay$. If a barge is unloaded in a given day, the minimum percentage of inventory that can be unloaded is $\minDailyUnloadPct_s$. If a supply barge has less than $\minDailyUnloadPct_s$ remaining, then the inventory is not unloaded and the barge departs.

A set $K$ of intermediate storage tanks is available to hold raw material or a blend of raw materials. Each tank $k$ has an initial inventory volume $v_{k,0}$, minimum inventory level $v_k^{\min}$, and maximum inventory level $v_k^{\max}$. The initial quality level for chemical property $q$ is denoted by $\SPEC_{k,q,0}$. Each tank can be supplied by the supply nodes in $K_s$. If the mixture in a tank $k$ is used to fulfill production feed, the minimum percentage of feed that can originate from the tank is $\minTankFeedPct_k$.
The demand for production feed is known for each time period $t \in T$. This demand is characterized by a volume $\demand_t$, bounds $\spec_{t,q}^{\min}$ and $\spec_{t,q}^{\max}$ on the quality level of the chemical properties, and bounds $\specRat_{t,q_1,q_2}^{\min}$ and $\specRat_{t,q_1,q_2}^{\max}$ on certain ratios of chemical properties. Production feed is planned in campaigns or runs of consecutive days.  During a run, $\spec_{t,q}^{\min}$ and $\spec_{t,q}^{\max}$ remain constant, and we require the daily flow plan for this feed to remain constant.

In order to meet production feed $\demand_t$, each time period, the volume of material transferred needs to be determined:  
\begin{itemize}
\item volume to unload from supply node $s$ to tank $k$ during time $t$, $\INB_{s,k,t}$ and
\item volume to unload from tank $k$ to meet production feed in time $t$, $\FEED_{k,t}$.
\end{itemize}

In addition, the following assignment variables on these flows are needed to address various operational constraints:
\begin{itemize}
\item if supply barge $s$ is unloaded during time $t$, then $\gamma_{s,t} = 1$.
\item If tank $k$ provides mixture to meet production feed in time $t$,  then $\sigma_{k,t} = 1$.
\end{itemize}

To model our objective of meeting demand and unloading requirements, we maximize a weighted sum of the amount of demand met and product unloaded.

Additional operational constraints include the following. The number of times a barge can be unloaded is limited to two, and all unload operations for the barge must be performed within a range of $\maxUnloadTimeGap$ time periods (e.g., seven days). %
Flow from tanks to meet production feeds must be continuous for each campaign (or demand run).

   Conservation of flow must be maintained, such that the sum of the initial tank inventory level and the volume of supply that flows into the tank less the volume that flows out for production feed equals the final tank inventory level. The resulting levels of chemical properties from this flow must meet the specified composition requirements. Due the relatively small sizes of the storage tanks in this work, we allow storage tanks to both receive supply and provide material for production feed simultaneously. If this occurs in a given time period, for the sake of simplicity, we model that complete mixing of inflows occurs before any outflows within that time period. The problem and constraints are formally specified in Section~\ref{sec:Mathematical-Model}.

\section{Literature Review \label{sec_litReview}}

The blending and scheduling problems, also known as multi-period pooling problems, studied in the literature vary widely in network size, scope, objective, and detail. Such problems, similar to the standing pooling problem, are typically characterized as bilinear nonconvex MIQCP's due to the balance of specs (or properties) in mixing and splitting of product. This is referred to as `linear blending', and is often an approximation of what occurs in practice (as e.g. the mixing process may not finish completely before material is split into multiple streams). However, similar problems can also involve linear \cite{Li2009} representation for spec balance, or a more highly nonlinear \cite{Zhang2002,Singh2000,Misener2009,Cuiwen2012} representations for balance of some specs (`nonlinear blending', as in the extended pooling problem), depending on the nature of the operational network or the specifications to track.

The underlying resource networks can be small and separated into discrete layers, and can represent front-end acquisition, as in e.g. \cite{Mouret2010,Oddsdottir2013,Chen2014,Cerda2015,Castro2016} and the current work; back-end operations and sales, as in e.g. \cite{Cuiwen2012}; or the full process from acquisition to sales, as in \cite{Mouret2010,Xu2017}. The networks can also span larger, more complex supply chains, as in \cite{Pinto2000,Castillo2017}. The scope of the problem can include scheduling only, as in \cite{Moro2004,Mendez2006,Kolodziej2013b,Castro2016} and the current work, but can also extend to include more high-level acquisition \cite{Oddsdottir2013}, production planning decisions \cite{LUO2009,Lu2015,Torkaman2017}, or more detailed equipment control decisions \cite{Dias2016}. Some works, such as \cite{GutirrezLimn2014,GutirrezLimn2016}, integrate all three levels of planning, scheduling, and control in a single model. 

Common examples of objective functions used in similar problems include maximization of profits \cite{Reddy2004a,Reddy2004b,Li2007,Kolodziej2012,Cerda2015,Castro2015b,Lotero2016,Castillo2017,Xu2017} or gross margins \cite{Mouret2010,Oddsdottir2013,Castro2014a,Gao2015,Castro2016}, or minimization of operating costs \cite{Mouret2010,Li2009,Li2010,Li2011,Castro2014a,Castro2016}. The objective function used in \cite{deAssis2017} in particular is similar to that in this work, but includes additional penalties for performing mixing operations, missing the demand composition requirements, and for not respecting the maintenance requirements. \cite{Cuiwen2012} maximizes the yields of final products. 
Many problems have only a single type of supply source, which can be characterized either by discrete arrivals \cite{Mouret2010,Oddsdottir2013,Chen2014,Castro2014a,Castro2016} or continuous supply streams e.g. \cite{Li2009,Li2011,Cuiwen2012,Castro2015b,Gao2015,Lotero2016}. Some works require multiple types of supply sources. For example, in \cite{Reddy2004a,Reddy2004b,Li2007,Cerda2015,Xu2017}, there are both small barges with a single type of inventory and larger barges, referred to as very large crude carriers (VLCCs), that carry multiple packages with varied compositions. Some works explicitly include vessel scheduling decisions to be made at unloading points, such as docks \cite{Reddy2004a,Li2007}. Most works assume known supply and demand quantities or bounds, but some, such as \cite{Neiro2006,LUO2009,Cuiwen2012,Gupta2014,Chu2015,Dias2016,Ning2017}, explicitly take into account stochastic supply and demand.

Muti-period pooling problems classically come in the forms of P, Q, or PQ formulations.  
The Q formulation is source-based, the P formulation is spec-based, and the PQ is a hybrid of the two. That is, the Q formulation tracks compositions based on origin fractions or volumes, while the P formulation tracks compositions explicitly based on the concentrations of the physical qualities to track.   See~\cite{Gupta2014} for a discussion of these models.
Some examples of spec-based formulations can be found in \cite{Cuiwen2012,Chen2014,Castro2015b}, while examples of source-based formulations can be found in \cite{Reddy2004a,Mouret2010,Gao2015,Cerda2015}. Comparing to the traditional pooling problem, a spec-based formulation mirrors the traditional P-formulation while a source-based formulation mirrors the Q or PQ-formulations. This choice of representation affects both the required number of bilinear terms and the MILP tightness of the formulation. Overall, the number of bilinear terms for a source-based formulation scales with the number of mixing or splitting nodes in the network and the number of distinct source compositions, while the number of bilinear terms for a spec-based formulation scales instead with the number of specs that require tracking. As such, this choice of representation can have a significant effect on both the tightness of the MIP and the number of required bilinear terms. While source-based formulations can be tighter \cite{Lotero2016}, they are only viable in situations where the number of distinct sources is relatively small. In this work, as many as 35 distinct source compositions can be represented within the scheduling horizon, rendering source-based formulations computationally expensive.

\subsection{Comparison to Similar Research}
We highlight and compare three research streams that are similar to this current work. Note that many differences between this work and others stem from the specialty chemical nature of our problem, combined with the properties of the production site for the application of interest.

Castro~\cite{Castro2016} studies a similar problem motivated by crude oil operations in refineries.  This problem is also addressed in~\cite{Lee1996,Mouret2009a,Mouret2010,Castro2014a}, among others.  The problem involves barge-based supply where the barges carry a single type of crude, arrive within given time intervals, and unload according to the arrival sequence.  The barges unload to storage tanks.  The material in the storage tanks are mixed into charging tanks that supply crude oil into charging tanks that supply crude oil distillation units (CDUs).

The similarities of the core problem include a similar network structure, a barge-based supply stream, and the capacity of each storage tank ($\sim$1-2 barges worth). Both works solve front-end operational scheduling problems and determine feed plans for downstream operations. In \cite{Castro2016}, such feeds are routed through crude oil distillation units (CDUs).  In this current work, the chemicals feed a single production feed.

However, the differences between the problems affect the modeling and solution approaches. In \cite{Castro2016}, the barges arrive at regular intervals and are unloaded in the order of arrival, with only one barge unloaded at a time. In the current work, multiple barges (up to a limit) may be unloaded simultaneously and the order of unloading needs to be determined. 
The shorter scheduling horizon in \cite{Castro2016} ($\le$15 days, $\le$3 barges) allows for solution over the entire scheduling horizon at once, while the long scheduling horizon (up to 365 days, $\sim$30 barges) in the current work necessitates a heuristic such as a rolling horizon scheme.
At the same time, the larger network in \cite{Castro2016} reduces the length of the scheduling horizon over which the problem can be solved within a reasonable time. A significant observation of \cite{Castro2016} is that, for the problem considered, the choice of objective function significantly impacted the relative efficacy of the discrete-time and continuous-time formulations. Another key difference is that the production schedule for each CDU is not fixed in \cite{Castro2016}, but is chosen based on gross margins per crude or on operating costs. On the other hand, a CDU can be sourced from only a single charging tank at a time. 

In Castro's work, the source-based choice is convenient because there are few composition types to track. In the current work, the source-based formulation is prohibitive due to the large amount of different possible composition types, and a spec-based formulation is used instead. As mentioned in \cite{Castro2016}, when viable, source-based formulations yield tighter MILP relaxations \cite{Lotero2016}, and can yield better computational performance \cite{Castro2015b}.%

Castro \cite{Castro2015b} addresses a similar multi-period pooling problem with composition constraints on the flows to demand stream (tested on both per-flow and a mixture bases).  Similar to the solution approach used in \cite{Castro2016}, Castro employs an  iterated MILP-NLP approach, where the MILP stage is represented using base-10 NMDT. However, the utilized objective functions and formulations differ. In \cite{Castro2015b}, a discrete-time formulation is used, and both spec-based and source-based formulations are presented, and compared with commercial global solvers. The objective function is profit maximization as opposed to gross margin maximization or cost minimization. The problem instance tested has two source streams, up to four blending/storage tanks for which tank-to-tank transfer operations are allowed, and two demand streams, and are solved with a time horizon of up to four periods.

Oddsdottir \cite{Oddsdottir2013} addresses a similar network, with the same structure (barges to storage tanks to demand (CDUs)). The main differences in this respect involve barge availability (less restrictive time windows), the number of storage tanks (6), and the number of demand streams (2). Another common feature is the requirement of a much longer scheduling horizon (3 months) and in the flexibility in barge arrivals. However, the goal in \cite{Oddsdottir2013} is related to procurement planning instead of operational scheduling. As such, the choice of which barges to order and when is a problem decision, and a courser time-grid is used (three days per period), resulting in a total of 30 time periods to model. Additionally, in both works, multiple barges might be unloaded within a single time period (up to a limit), possibly contributing to the choice of a discrete-time formulation. Further, a given barge may unload to multiple storage tanks in a given time period. As with \cite{Castro2016}, however, the limited number of crude types in the system in \cite{Oddsdottir2013} allow for a reasonably compact source-based formulation, which would be unwieldy in the current work.

\subsection{Related Solution Methodologies}

Rolling planning aside, the solution approach selected in \cite{Oddsdottir2013} is comparable to the approach in \cite{Castro2016}, but with some key differences. As in \cite{Castro2016}, a source-based formulation is used and composition constraints are formulated with respect to crude volumes instead of crude volume fractions. However, instead of discretizing, \cite{Oddsdottir2013} applies a MILP approximation to the problem in which the constraints enforcing correct proportions for the outflow composition are reformulated in a manner that allows composition discrepancy in the outflows. A nonconvex MIQCP stage is then applied to resolve the composition discrepancy, yielding near-optimal feasible solutions in a much shorter time frame than the full nonconvex MIQCP: Running in GAMS with BARON, over 20 datasets, the average solution gap for this PRONODIS method (compared to the full nonconvex MIQCP) was 1.1\%, with an average computation time of 20s (compared to over 13000s for the full nonconvex MIQCP).

A wide variety of formulations and approaches have been developed to tackle these difficult scheduling problems, with varying time-scheduling formulations, spec representations, and algorithms to deal with the problem nonlinearity, which usually incorporate some form of MILP approximation to the basic nonconvex MIQCP.

A number of continuous-time and discrete-time scheduling formulations have been presented. %
The choice of time-formulation can have a significant impact on problem performance, as showcased in works such as \cite{Reddy2004b} and \cite{Mouret2009a,Mouret2010}. The preferred scheduling formulation can depend on the type of objective function used, as observed in \cite{Castro2016}, as well as on the types of operational constraints to be considered. 
Compared to continuous-time formulations, discrete-time formulations frequently yield simpler representations, with a trade-off of more integer variables. Even so, problems with certain types of operational constraints, such as variable flow rates and constraints on the number of simultaneous operations (such as barge unloads), can be significantly easier to model and solve in discrete-time (as seen in e.g. \cite{Castro2016}). A discrete-time formulation is also used in our work.

A number of solution methodologies have emerged to deal with the handling of the nonlinear nature of these scheduling problems, particularly over long scheduling horizons. Such methods include single-step \cite{Wenkai2002,Mouret2010,Oddsdottir2013,Castro2014a,Castro2015b,Castro2016} and iterated MILP-NLP \cite{Kelly2003,Karuppiah2008,Cerda2015,deAssis2017,Xu2017} approaches, which have emerged as some of the most efficient methods to ensure small optimality gaps for relatively short scheduling horizons. Other methods include MILP-(smaller MINLP) iterations \cite{Oddsdottir2013,Lotero2016} and iterative MILP-refinement methods \cite{Reddy2004a,Reddy2004a,Mendez2006,Li2007,Kolodziej2013b}.

For longer scheduling horizons, rolling planning methods are often used to limit the number of integer variables. Such models typically yield good feasible solutions far faster than the full model, with far better scaling of computation times with the scheduling horizon. 
In general, a rolling horizon model operates on a discrete time grid by representing in detail only the `present' segment $\{t^h, \dots, t^h+H^{\roll}-1\}$ of the scheduling horizon $\{1,\dots,H\}$ at each iteration, then tanking a step forwards in time of size $t^{\step} \le t^h$ for the next iteration, freezing the decisions made in the interval $\{t^h,\dots,t^h+t^{\step}-1\}$ and proceeding until the full horizon is reached.

In some rolling planning models, such as those in \cite{Silvente2015,Cuiwen2012,Oddsdottir2013,Stolletz2014,Lu2015} and heuristic H2 in \cite{Torkaman2017}, all aspects of the `past' decisions in $\{1,\dots,t^h-1\}$ are frozen; in others, such as \cite{Cerda2015} and heuristics H1, H3, and H4 in \cite{Torkaman2017}, only certain variables are frozen (most commonly the binary/integer decisions). The treatment of the `future' portion of the scheduling horizon $\{t+H^{\roll},\dots,H\}$ varies considerably between models: Some models exclude the `future' from the current model entirely, as in \cite{Silvente2015,Cuiwen2012,Stolletz2014,Marquant2015,Cerda2015}, whereas some relax complicating constraints \cite{LUO2009,Torkaman2017} and/or treat integer variables as continuous \cite{Lu2015,Torkaman2017}. Some works, such as \cite{Cuiwen2012,Oddsdottir2013,Champion2017,Silvente2018}, deal with uncertain events by updating uncertain data or events at each (or certain) rolling planning steps(s). For bi-level integrated planning and scheduling models, it is common to solve the scheduling subproblem only for the `present' segment, while including (a subset of) the future segments in the planning subproblem, as is done in \cite{Li2010,Zondervan2014,Silvente2015,Silvente2018}. Some works use an alternate framework based on supply arrivals to determine the time steps for each rolling iteration, as in \cite{Reddy2004a,Reddy2004b,Li2007,Cerda2015}; \cite{Li2007} in particular allows for some backtracking to correct spec quality discrepancies in the current plan.

Finally, many ideas for MILP approximations to deal with bilinear terms have emerged in recent years, especially for the pooling problem. Some methods are either binary-free, or add a number of binary variables that scales linearly with the desired level of precision. These include various linearization techniques \cite{Joly2003,Moro2004}, piecewise-linear approximation \cite{Pham2009,Gao2015}, outer-approximation methods \cite{Xu2017}, and McCormick or piecewise McCormick envelopes \cite{McCormick1976,Bergamini2005,Mouret2009b,Castro2015a,Castro2015b,deAssis2017}. \cite{Gounaris2009} in particular gives a numerical comparison for fifteen different piecewise McCormick-related schemes, split into three classes of schema. For other methods, often utilizing digit-based representations for a single variable, the number of binaries scales logarithmically with the desired precision. Such methods include multi-parametric disaggregation \cite{Teles2011,Kolodziej2013a,Castro2014a,Castro2014b,Castro2015b,Castro2015c,Castro2016} and other techniques \cite{Vielma2009,Misener2011,Gupte2013,Kolodziej2013b,Gupte2016}. Approximation via generalized disjunctive programming (GDP) is another, far more expressive, approach that has been used to tackle pooling-type problems. Castro and Grossmann et. al. have derived a MILP-based formulation for variants of multi-parametric disaggregation as a disjunctive program \cite{Grossmann2011,Kolodziej2013a,Castro2014a,Castro2015b,Castro2015c}, and have also applied it separately \cite{Ruiz2010,Castro2016,Lotero2016}. In this work, we choose to directly discretize the inventory specs of each tank using a formulation based on normalized multi-parametric disaggregation (NMDT), as set forth in \cite{Castro2015c}, yielding a relatively simple and highly efficient model in conjunction with a rolling planning approach. 

We were unable to find any other work in the literature which combines a version of the multi-parametric disaggregation technique with a version of rolling planning; this combination is a key novelty of this work. The primary results of this work include the speed of the method achieved while yielding near-optimal solutions without compromising the qualities of the specs in the demand feeds, combined with the incorporation of constraints on the ratios of specs in the demand feeds.

\section{Mathematical Model}
\label{sec:Mathematical-Model}

We describe a nonconvex MIQCP model in this section that introduces variables to handle the constraints.  Data in the problem is written with an underscore, continuous variables are written with lower-case letters, binary variables are written with Greek letters, and sets are written with capital letters.
\subsection{Terminology}
\begin{tabular}{ll}
   Run  & Production run, consisting of a number of consecutive days for which the production feed\\
        & \hspace{.2cm} must be constant in terms of volume.\\
   Feed  & Flow from a storage tank to the downstream production processes. \\
   Mixing & The mixing of material from a number of different sources to a single destination.\\
   Splitting & The splitting of material from a single source to a number of different destinations.\\
   Spec & Concentration of a certain type of chemical property.
\end{tabular}
\subsection{Sets and Parameters}

\underline{Sets}\\
\begin{tabular}{ll}
	$S$ &: Set of supply nodes (barges with raw material), indexed by $s$.\\
	$S_k$ &: Set of supply barges allowed to unload to tank $k \in  K$.\\
    $S_t$ &: Set of supply barges available during time period $t \in  T$.\\
    $K$ &: Set of tanks, indexed by $k$.\\
    $K_s$ &: Set of tanks to which supply barge $s \in S$ is allowed to unload. \\
    $Q$ &: Set of chemical properties, indexed by $q$.\\
    $Q_{\rat}$ &: Pairs of specification for which ratio bounds are to be enforced.\\
    $R$ &: Set of non-overlapping demand runs (consecutive days for which the production feed is constant). \\
    $T$ &: Set of time periods (days), such that $T$ $\defeq\{1, 2, \dots, H\}$.\\
    $T_r$ &: Time periods corresponding to run $r$ where runs do not overlap, such that $T_r =  \{t^{i}_r, t^{i}_r+1, \dots, t^f_r\} \subseteq T$. \\
    $T_s$ &: Time periods during which supply $s \in S$ can be unloaded, such that  $T_s =  \{t^{\low}_s,t^{\low}_s+1,\dots,t^{\upp}_s\}$. \\
    $T_d$ &: Time periods where there is some demand, such that  $T_d=  \bigsqcup_{r \in R} T_r $.\\
\end{tabular}

\underline{Parameters}\\
\begin{longtable}{ll}
    $\misInbPenalty_s$ &: Penalty cost per volume of not unloading supply from supply node $s$.\\
    $\misPenalty_t$ &: Penalty cost per volume of not meeting production feed volume for time period $t$.\\
    $\demand_t$ &: Demand to meet in time step $t \in T$. Must be constant throughout each run.\\
    & \hspace{.2cm} Thus, $\demand_t = \demand_{t'}$ for all $t,t' \in T_r$ for $r \in R$. Also, $\demand_t = 0$ for $t \notin T_d$.\\
    $\spec_{s,q}$ &: Value of chemical characteristic  $q \in Q$ in supply source $s \in S$.\\
    $\spec_{k,q,0}$ &: Initial value of chemical characteristic $q \in Q$ in tank $k \in K$.\\
    $\spec^{\low}_{q,t}, \spec^{\upp}_{q,t}$ &: Bounds on the value of chemical characteristic $q \in Q$ for the demand to be filled \\
    & \hspace{.2cm} in time step $t \in T$.\\
	$H$ &: Scheduling time horizon (implicit in the definition of $T$).\\
    $\specRat^{\low}_{q_1,q_2,t}, \specRat^{\upp}_{q_1,q_2,t}$ &: Bounds on the ratio of chemical characteristics $\SPEC_{q_1}$ and $\SPEC_{q_2}$,$(q_1,q_2) \in Q_{\rat}$ for the demand \\
    & \hspace{.2cm} to be filled in time step $t \in T$.\\
    $\minTankFeedPct_k$ &: Minimum percentage of feed each day that can originate from \\
    & \hspace{.2cm} tank $k$ (if used).\\
    $\maxUnloadsPerDay$ &: Maximum number of supply barges that can be unloaded each day.\\
    $\minDailyUnloadPct_s$ &: Minimum percentage of inventory from supply barge $s$.\\
    & \hspace{.2cm} that can be unloaded each day that unloading occurs for $s$.\\
    $\maxUnloadTimeGap$ &: Maximum permissible gap between the first and last unloads \\
    & \hspace{.2cm} for each supply barge.\\
    $\invInb_s$ &: Volume of supply available for unloading from supply node $s \in S$. \\
    $\invInit_{k,0}$ &: Initial inventory volume of tank $k \in K$.\\
	$[\underbar{v}^{\min}_k,\underbar{v}^{\max}_k]$ &: Bounds on permissible inventory for each tank $k \in K.$
\end{longtable}

\underline{Variables} (All variables are non-negative)\\
\begin{longtable}{lll}
	Assignment: & $\gamma_{s,t}$ &: $ = \begin{cases}
		1: & \text{supply barge } s \text{ is unloaded in time step } t\\
        0: & \text{otherwise}
    \end{cases}$ : $s \in S, t \in T$\\
    ~ & $\sigma_{k,t}$ &: $ = \begin{cases}
		1: & \text{tank } k \text{ provides material for production feed in time step }t\\
        0: & \text{otherwise}
    \end{cases}$ : $k \in K, t \in T$ \\
    Transfer: & $\INB_{s,k,t}$ &: Volume of raw material to transfer from supply $s \in S$ to\\
    & & \hspace{0.2cm} tank $k \in K$ in time step $t \in T$\\
    & $\FEED_{k,t}$ &: Volume of material to feed to demand from node $k \in K$ \\
    & & \hspace{0.2cm} in time step $t \in T$\\
    Resource: & $\INV_{k,t}$ &: Volume of inventory in tank $k \in K$ at the end of time step $t \in T$\\
    & $\SPL_{k,t}$ &: Volume of inventory in tank $k \in K$ after material is supplied from supply barges,\\
    & & \hspace{0.2cm} but before feed operations, in time step $t \in T$\\
    Demand: & $\vd_t$ &: Total volume fed to production in time step $t \in T$\\
    Tank specs: & $\SPEC_{k,q,t}$ &: Value of chemical characteristic $q \in Q$ in tank $k \in K$ at the end of time step $t \in T$\\
    Final specs: & $\SPECd_{q,t}$ &: Value of chemical characteristic $q \in Q$ at production at the end of time step $t \in T$\\
    Unload time: & $\MINUNLTIME_{s}$ &: Earliest unload date for supply barge $s$\\
    ~ & $\MAXUNLTIME_{s}$ &: Latest unload date for supply barge $s$\\
    Slack: & $\MISINB_s$ &:  Volume of material that is not unloaded from supply barge $s \in S$\\
    ~ 	  & $\MIS_t$ &: Volume of production demand missed in time step $t \in T$\\
\end{longtable}

\underline{Objective:} Our objective is to minimize the amount of material that is not unloaded from barges ($\MISINB_s$) and the demand that is unmet ($\MIS_t$).  To accomplish this, we maximize the sum of the value of the raw material unloaded from the barges and the value of the production feed demand that is met.  The amount of material unloaded for a barge $s$ is exactly $\invInb_s - \MISINB_s$, and the amount of demand met at time $t$ is $\demand_t - \MIS_t$.  Hence, we have   
	\begin{equation}
    	\begin{array}{ll}
        	\max & \dsum_{s \in S} \misInbPenalty_s \cdot (\invInb_s - \MISINB_s) + \sum_{t \in T} \misPenalty_t \cdot (\demand_t - \MIS_t)
        \end{array}
    \end{equation}

\subsection{Constraints}
The following constraints limit the solution space of the assignment, blending, and mixing problem.  In the following discussion, note that variables $\INB,\FEED$ correspond to flows, while the variables $v$ correspond to stationary volumes at certain times including volumes on barges $v_s$, volumes in tanks $v_{k,t}$, and volumes required for production feed $v_t$.  
\\

\textit{\underline{Initial Conditions}}: The volume of material in the tanks and the chemical characteristics of the material are initialized.  Specifically,  \eqref{eq_invInit-tnk} enforces that the initial volumes in each storage tank are correctly represented, and  \eqref{eq_specInv} enforces that the initial chemical characteristics in each storage tank are specified.  
\begin{subequations}
\begin{align}
    \INV_{k,0} &= \invInit_{k,0} & \,& k \in K \label{eq_invInit-tnk}\\
    \SPEC_{k,q,0} &= \spec_{k,q,0} & \,& k \in K, q \in Q \label{eq_specInv}
\end{align}
\end{subequations}

\textit{\underline{Flow}}: 
For these flow constraints, we assume that  inbound supply material is unloaded and blended in the tanks before any material is drawn from the demand tanks to meet production feed on any given day. 
The model assumes that in the first half of a time period, barges unload to tanks \eqref{eq_invinb}, adding to the volume previously in the tank and creating an intermediate volume $\SPL_{k,t}$.  In the second half of the time period, material is unloaded from the tanks \eqref{eq_outflow} to support production \eqref{eq_invProduction}, while within required volume bounds at production \eqref{eq_invlb} and \eqref{eq_invub}. 
\begin{subequations}
\begin{align}
	\dsum_{s \in S_k} \INB_{s,k,t} + \INV_{k,t-1} &= \SPL_{k,t}%
	& \,& k \in K, t \in T \label{eq_invinb}\\
		\SPL_{k,t}%
		-  \FEED_{k,t} &= \INV_{k,t} & \,& k \in K, t \in T \label{eq_outflow}\\
	\dsum_{k \in K} \FEED_{k,t} &= \vd_{t} & \,& t \in T \label{eq_invProduction}\\
	\INV_{k,t} &\ge \underbar{v}^{\min}_k & \,& k \in K, t \in T \label{eq_invlb}\\
	\SPL_{k,t} &\le \underbar{v}^{\max}_k & \,& k \in K, t \in T \label{eq_invub}
\end{align}
\end{subequations}

\textit{\underline{Chemical Characteristics}}: \Nonlinear The following constraints capture the blending and mixing of chemical characteristics.  Specifically, \eqref{eq_infspec} ensures the characteristics of the material in a tank at the end of the previous time period are combined with the material from supply barges to create a mix. Expression \eqref{eq_feedspecinv} determines the characteristics of the material provided from tanks to meet production feed.   
\begin{subequations}
\begin{align}
	\SPEC_{k,q,\thalf}  \cdot \SPL_{k,\thalf}
	&= \displaystyle \sum_{s \in S_k} \spec_{s,q} \cdot \INB_{s,k,t}   + \SPEC_{k,q,t-1} \cdot \INV_{k,t-1} & \,& k \in K, q \in Q, t \in T\label{eq_infspec}\\
	 \SPECd_{q,t} \cdot \vd_{t} &= \dsum_{k \in K} \SPEC_{k,q,t} \cdot \FEED_{k,t}  & \,& q \in Q, t \in T \label{eq_feedspecinv}
\end{align}
\end{subequations}
 Thus, while supply unloads and production feeds can occur in the same time period (day), they are not considered to occur simultaneously.\\

\textit{\underline{Demand}}: The following constraint set captures the production feed volume that is not met each time period,  $\MIS_t$.
\begin{align}
	\displaystyle \vd_{t} &= \demand_t - \MIS_t &  t \in T \label{eq_demand}
\end{align}
\textit{\underline{Feed Characteristics}}: Both the chemical characteristics \eqref{eq_feedspec} and ratios of characteristics  \eqref{eq_feedspecrat} need to be within the required bounds. Note that in practice for \eqref{eq_feedspecrat} the denominator would be multiplied through, as numerical issues can arise when the denominator is small. 
\begin{subequations}
\begin{align}
	\spec^{\low}_{q,t} &\le  \SPECd_{q,t} \le \spec^{\upp}_{q,t} & \,& t \in T,~q \in Q \label{eq_feedspec}\\
	 \specRat^{\low}_{q_1,q_2,t} &\leq \frac{\SPECd_{q_1,t}}{\SPECd_{q_2,t}} \leq \specRat^{\upp}_{q_1,q_2,t} & \, & t \in T,~(q_1,q_2) \in Q_{\rat}  \label{eq_feedspecrat}
\end{align}
\end{subequations}

\textit{\underline{Total Inflow}}: Ideally, the volume of supply on each barge is fully unloaded.  Expression \eqref{eq_tot_inflow} captures the volume of supply that is not unloaded ($\MISINB_s$).  In addition, \eqref{eq_correcttime_inflow} ensures no material is unloaded from a supply barge when the barge is not available.
 \begin{subequations}
\begin{align}
	\displaystyle \sum_{k \in K_s} \sum_{t \in T_s} \INB_{s,k,t} &= \invInb_s - \MISINB_s  & &\quad  \, && s \in S \label{eq_tot_inflow}\\
	\INB_{s,k,t} &= 0  & &\quad \, &&  s \in S, k \in K_s, t \in T \backslash T_s \label{eq_correcttime_inflow}
\end{align}
 \end{subequations}

\textit{\underline{Continuous Feed During Runs}}: The volume of supply provided by each tank to meet production feed needs to remain constant during a production run. %
\begin{align}
	\displaystyle \FEED_{k,t} &= \FEED_{k,t-1} &  r \in R, k \in  K,~t \in T_r \backslash \{t^{i}_r\} \label{eq_contfeed}
\end{align}

\textit{\underline{Tank Feed Share Constraints}}:  If production feed is supplied by tank $k$ during time period $t$ (such that $\sigma_{k,t} = 1$), then the volume of production feed ($\FEED_{k,t}$) must be at least $\minTankFeedPct_k$ of the demand $\demand_t$ at time $t$, as enforced by \eqref{eq_feedshare_lb}.  If $\sigma_{k,t} = 0$, then \eqref{eq_feedshare_ub} ensures no feed is provided by tank $k$ for time $t$.
\begin{subequations}
\begin{align}
	&\displaystyle \FEED_{k,t} \ge (\minTankFeedPct_k \cdot \demand_t) \cdot \sigma_{k,t} & \,&  k\in K, t \in T \label{eq_feedshare_lb}\\
    &\displaystyle \FEED_{k,t} \le \demand_t \cdot \sigma_{k,t}  & \,& k \in K, t \in T \label{eq_feedshare_ub}
\end{align}
\end{subequations}
\textit{\underline{Supply Limitations}}: Although it is preferred that barges are completely unloaded in one day, due to tank storage capacity, this may need to be done over several  days.  Hence, we constrain the number of unloads to be bounded by $\bargeNumUnloads$, which are $2$ in all of our instances \eqref{eq_barge_num_unloads}. The number of times any barge is unloaded within a time period is also limited \eqref{eq_day_num_unloads}. Expression \eqref{eq_tighten_correctperiod_unloads} serves as a tightening constraint to ensure that barges are not unloaded when they are not available. If a supply barge is not unloaded during a time period (such that $\gamma_{s,t}$ is zero), then no  volume is unloaded from the barge \eqref{eq_unl_ub}.  
\begin{subequations}
\begin{align}
	\displaystyle \sum_{t \in T_s} \gamma_{s,t} &\le \bargeNumUnloads & \,& s \in S \label{eq_barge_num_unloads}\\
    \gamma_{s,t} &= 0 & \,& s \in S, t \in T\backslash T_s \label{eq_tighten_correctperiod_unloads}\\
    \displaystyle \sum_{s \in S} \gamma_{s,t} &\le \maxUnloadsPerDay & \,& t \in T \label{eq_day_num_unloads}\\
    \INB_{s,k,t} &\le   \invInb_{s}\cdot \gamma_{s,t} & \,& s \in S, t \in  T,~k \in K_s \label{eq_unl_ub}
\end{align}
\end{subequations}

\textit{\underline{Unload Inventory Percentage Constraints}}:   \eqref{eq_minunl_pct} enforces that the percent of inventory unloaded per supply barge remains above required minimum $\minDailyUnloadPct_s$ each day, if unloading is to occur that day. 
\begin{align}
	\displaystyle \sum_{k \in S_k} \INB_{s,k,t} &\ge  ( \invInb_s \cdot \minDailyUnloadPct_s)  \cdot \gamma_{s,t} & s \in S, t \in  T \label{eq_minunl_pct}
\end{align}

\textit{\underline{Unload Time Gap Constraints}}: In an effort to minimize the cost associated with barges remaining in the area for an extended time, these constraints limit the time duration for a supply barge to remain in the unloading area.  To accomplish this, constraints are included to limit the interval $[\MINUNLTIME_s,\MAXUNLTIME_s]$.  Expression \eqref{eq_minunltime} ensures that  $\MINUNLTIME_s$ is no later than the first period that any inventory is unloaded for barge $s$. Expression \eqref{eq_maxunltime} enforces that $\MAXUNLTIME_s$ is no earlier than the last period that any inventory is unloaded for barge $s$. Finally, \eqref{eq_unlwindow} enforces that the difference between the first and last unloading times for each barge is no larger than the allotted time gap.  Although $\MINUNLTIME_s$ may assume a value earlier than the actual earliest unload time and $\MAXUNLTIME_s$ may assume a value after the last unload time, if the interval is \textit{shorter} than the required length, then the actual interval will also be shorter.  Thus, this set of constraints is sufficient. 
\begin{subequations}
\begin{align}
	&\MINUNLTIME_s \le t \cdot \gamma_{s,t}  + H(1 - \gamma_{s,t})& \,& s \in S, t \in  T \label{eq_minunltime}\\
	&\MAXUNLTIME_s \ge t \cdot \gamma_{s,t} -H(1-\gamma_{s,t})& \,& s \in S, t \in  T \label{eq_maxunltime}\\
    &\MAXUNLTIME_s - \MINUNLTIME_s \le \maxUnloadTimeGap \label{eq_unlwindow} & \,& s \in S
\end{align}
\end{subequations}
All of the constraints described above are linear except for the chemical characteristic constraints described by \eqref{eq_infspec},\eqref{eq_feedspecinv}, which contain bilinear terms. MILP approximations for these constraints are formulated in Section~\ref{ssec_MILP-approx}. 

\subsection{Reformulation without spec variables $\SPECd_{q,t}$ \label{ssec_prod}}
For the implementation of this model, we adjust the model to eliminate terms consisting of only a single $\SPECd$ variable from the model, restricting the model to product terms $\SPEC \cdot v$ or $\SPECd \cdot v$, referred to as spec-volume terms. In this way, after discretization, we no longer have need to explicitly keep track of these specs at all, relying solely on these spec-volume product terms. In addition, we eliminate the $\SPECd_{q,t}$ and $\vd_{t}$ variables. 

The initial constraint \eqref{eq_specInv} is  replaced with
    \begin{align}
         \SPEC_{k,q,0} \cdot \INV_{k,0} &= \spec_{k,q,0} \cdot \invInit_{k,0} & k \in K,~ q \in Q \label{eq_specinv_2}\tag{\ref{eq_specInv}-prod}
         \end{align}
By multiplying \eqref{eq_outflow} by $\SPEC_{k,q,t}$ we obtain
\begin{align}
     \SPEC_{k,q,t} \cdot \SPL_{k,t} - \SPEC_{k,q,t} \cdot \FEED_{k,t} &= \SPEC_{k,q,t} \cdot \INV_{k,t} & k \in K,~ t \in T
\label{eq_specOutflow} \tag{\ref{eq_outflow}-prod}
\end{align}
Note that the $\SPL_{k,t}$ variables could also be eliminated via the constraint \eqref{eq_outflow}. However, as this elimination choice can impact the quality of the MILP approximation, we defer this discussion to a later section.

We reformulate the constraints \eqref{eq_feedspec},\eqref{eq_feedspecrat} by multiplying through by volume (and any denominators) to obtain the following:
\begin{subequations}
\begin{align}
	\spec^{\low}_{q,t} \cdot \vd_{t} &\le  
	\SPECd_{q,t} \cdot \vd_{t} \le 
	\spec^{\upp}_{q,t} \cdot \vd_{t} & 
	\,& t \in T,~q \in Q
	\label{eq_feedspec_2}\tag{\ref{eq_feedspec}-prod}\\
	 \specRat^{\low}_{q_1,q_2,t} \cdot\SPECd_{q_2,t} \cdot \vd_{t}   &\leq \SPECd_{q_1,t}\cdot \vd_{t} \leq 
	\specRat^{\upp}_{q_1,q_2,t} \cdot\SPECd_{q_2,t} \cdot \vd_{t}
	 & \, & t \in T,~(q_1,q_2) \in Q_{\rat}  \label{eq_feedspecrat_2}\tag{\ref{eq_feedspecrat}-prod}
\end{align}
\end{subequations}
Finally, we substitute the definitions of $v_{t}$ and $\SPEC_{k,q,t} \cdot v_{t}$ in \eqref{eq_invProduction} and \eqref{eq_feedspecinv} into constraints \eqref{eq_demand}, \eqref{eq_feedspec_2},\eqref{eq_feedspecrat_2} to obtain:
\begin{subequations}
\begin{align}
      \dsum_{k \in K} \FEED_{k,t} &= \demand_t - \MIS_t & \, & t \in T \label{eq_demand_2}
	  \tag{\ref{eq_demand}-sub}\\
	\spec^{\low}_{q,t} \cdot \dsum_{k \in K} \FEED_{k,t} &\le  
	\dsum_{k \in K} \SPEC_{k,q,t} \cdot \FEED_{k,t} \le 
	\spec^{\upp}_{q,t} \cdot \dsum_{k \in K} \FEED_{k,t} & 
	\,& t \in T,~q \in Q \label{eq_feedspec_3}\\%\tag{\ref{eq_feedspec_2}-sub}\\
%	%
	\specRat^{\low}_{q_1,q_2,t}\cdot \dsum_{k \in K} \SPEC_{k,q_2,t} \cdot \FEED_{k,t}   &\leq 
	 \dsum_{k \in K} \SPEC_{k,q_1,t} \cdot \FEED_{k,t} \leq 
	 \specRat^{\upp}_{q_1,q_2,t}\cdot  \dsum_{k \in K} \SPEC_{k,q_2,t} \cdot \FEED_{k,t}   
	 & \, & t \in T,~(q_1,q_2) \in Q_{\rat}  \label{eq_feedspecrat_3}%\tag{\ref{eq_feedspecrat_2}-sub}
\end{align}
\end{subequations}
To summarize, the MINLP-Mixing reformulation is the initial model, but with \eqref{eq_specInv}, \eqref{eq_outflow}, \eqref{eq_invProduction}, \eqref{eq_feedspecinv}, \eqref{eq_demand}, \eqref{eq_feedspec}, and \eqref{eq_feedspecrat} replaced by \eqref{eq_specinv_2}, \eqref{eq_specOutflow}, \eqref{eq_demand_2}, \eqref{eq_feedspec_3},\eqref{eq_feedspecrat_3}.

\subsection{Model MINLP-SPLIT: Volume-based reformulation \label{ssec_split}}
The model presented above can be reformulated to keep track of volumes of chemical characteristics instead of percentages $\SPEC_{k,q,t}$, removing $\SPEC_{k,q,t}$ from the model. This reformulation results in a linear representation for the mixing constraints in \eqref{eq_feedspec_3},\eqref{eq_feedspecrat_3},\eqref{eq_infspec} and modified flow balance constraint in \eqref{eq_specOutflow}. However, there is now a nonlinear representation for the splitting process, in which some volume is split into multiple parts (e.g. future inventory and demand feed). In the original model, the definition and usage of $\SPEC_{k,q,t}$ ensures that the outflow characteristics are correct. However, after reformulating without $\SPEC_{k,q,t}$, this is not longer the case, so we must ensure that the outflow characteristics are correct.

To enact the reformulation, we make the following substitutions:
\begin{subequations}
\begin{align}
    \SPEC_{k,q,t} \cdot \INV_{k,t} & = \SVOLINV_{k,q,t} & \, & k \in K, q \in Q, t \in T\\
    \SPEC_{k,q,t} \cdot \SPL_{k,t} & = \SVOLSPL_{k,q,t} & \, & k \in K, q \in Q, t \in T\\
    \SPEC_{k,q,t} \cdot \FEED_{k,t} & = \SVOLFEED_{k,q,t} & \, & k \in K, q \in Q, t \in T
\end{align}
\end{subequations}
Thus, \eqref{eq_specOutflow}, \eqref{eq_infspec}, \eqref{eq_feedspec_3} and \eqref{eq_feedspecrat_3} are written as
\begin{subequations}
\begin{align}
     \SVOLINV_{k,q,t} &= \SVOLSPL_{k,q,t} - \SVOLFEED_{k,q,t} & \,& k \in K, t \in T
     \label{eq_specOutflow_2}\\% \tag{\ref{eq_specOutflow}-vol}\\
	\SVOLSPL_{k,q,\thalf} &= \displaystyle \sum_{s \in S_k} \spec_{s,q} \cdot \INB_{s,k,t} + \SVOLINV_{k,q,t-1} 
	& \,& k \in K, q \in Q, t \in T\label{eq_infspec_2}\tag{\ref{eq_infspec}-vol}\\
	\spec^{\low}_{q,t} \cdot \dsum_{k \in K} \FEED_{k,t} &\le  
	\dsum_{k \in K} \SVOLFEED_{k,q,t} \le 
	\spec^{\upp}_{q,t} \cdot \dsum_{k \in K} \FEED_{k,t} & 
	\,& t \in T,~q \in Q \label{eq_feedspec_4}\\%\tag{\ref{eq_feedspec_3}-vol}\\
	\specRat^{\low}_{q_1,q_2,t}\cdot \dsum_{k \in K} \SVOLFEED_{k,q_2,t} &\leq 
	 \dsum_{k \in K} \SVOLFEED_{k,q_1,t} \leq 
	 \specRat^{\upp}_{q_1,q_2,t}\cdot  \dsum_{k \in K} \SVOLFEED_{k,q_2,t} 
	 & \, & t \in T,~(q_1,q_2) \in Q_{\rat}  \label{eq_feedspecrat_4}%\tag{eq_feedspecrat_3-vol}
\end{align}
\end{subequations}
In addition, the initial condition \eqref{eq_specinv_2} is replaced with
    \begin{align}
         \SVOLINV_{k,q,0} &= \spec_{k,q,0} \cdot \invInit_{k,0} & k \in K,~ q \in Q \label{eq_specinv_3}%\tag{\ref{eq_specinv_2}-prod}
     \end{align}
Now, to enforce the requirement that the chemical characteristics sent to production feed match match the chemical characteristics in the tank; that is:
\begin{subequations}
\begin{align}
    \SPEC_{k,q,t} = \dfrac{\SVOLSPL_{k,q,t}}{\SPL_{k,t}} =  \dfrac{\SVOLFEED_{k,q,t}}{\FEED_{k,t}}
\end{align}
\end{subequations}
we add the following nonlinear constraint
\begin{subequations}
\begin{align}
    \SVOLSPL_{k,q,t} \cdot \FEED_{k,t} &= \SVOLFEED_{k,q,t} \cdot \SPL_{k,t}  & \, & k \in K,~ q \in Q,~ t \in T
    \label{eq_split}
\end{align}
\end{subequations}
To summarize, the MINLP-Splitting model is the initial model, but with \eqref{eq_specInv},\eqref{eq_outflow},\eqref{eq_invProduction},\eqref{eq_infspec},\eqref{eq_feedspecinv},\eqref{eq_demand},\eqref{eq_feedspec},\eqref{eq_feedspecrat} replaced by \eqref{eq_specinv_3}, \eqref{eq_specOutflow_2}, \eqref{eq_demand_2}, \eqref{eq_feedspec_4}, \eqref{eq_feedspecrat_4}, and \eqref{eq_split}.

\section{MILP Approximation Methods}
We consider approaches of changing the nonlinear portions of the model in Section~\ref{ssec_prod} to an MILP to more efficiently solve the problem and more easily interact with a rolling planning approach.  The core change is that we will replace all products $\INB\cdot f$ and $\FEED\cdot f$ with approximations. The new models, given sufficiently accurate approximations, will produce feasible flows $\INB$ and $\FEED$ that completely determine the actions of the system.  

For a spec variable $f \in [l,u]$ we describe it by a discete part $\dot{f} \in \{l, l + \epsilon, l+2\epsilon, \dots, l+ \zeta \epsilon\}$ and a continuous part $\Delta f \in [0,\epsilon]$, where we choose $\epsilon$ such that $u - \epsilon < 1 + \zeta \epsilon \leq u$.  The discrete part $\dot{f}$ will be modeled with binary variables in a way known as Normalized Multi-Parametric Disaggregation Techniques (NMDT). For the products with $\Delta f$, we will consider two options - restricting the value to a midpoint or using the McCormick relaxation.  That is, the two main ideas considered are the models:  
 \begin{align}
x \cdot \tilde f &= \underbrace{x  \cdot \dot{f}}_{\text{NMDT}} + \underbrace{x \cdot \Delta f}_{\Delta f \leftarrow \frac{\hat \epsilon}{2}} &
\label{approx-model}\tag{Center}\\
x \cdot \tilde f &= \underbrace{x  \cdot \dot{f}}_{\text{NMDT}} + \underbrace{x \cdot \Delta f}_{\text{McCormick}} &
\label{NMDT-model}
\tag{McCormick}
\end{align}
When using \eqref{approx-model}, we relax \eqref{eq_infspec} to maintain flexibility on inflow volumes. We do this by replacing the left-hand side with its upper and lower bounds with respect to $\Delta f$.
These choices, along with various relaxations of redundant constraints are described in detail in this section.

We will first describe common linearizations that we will consider, then describe our two competing models, and then have a discussion of rolling planning approaches.

  \subsection{Linearization}
  We define here two types of linearizations that we will consider.  First,   
the McCormick envelope is a 3-dimensional polytope $\mathcal{M}(x,\beta)$ that is the smallest convex set containing the equation $\xsf^\betasf =  x \cdot \beta$ with upper and lower bounds on $\beta$ and $x$. We will assume $0 \leq \beta \leq 1$ and $\underbar{x}^{\min} \leq x \leq \underbar{x}^{\max}$.  The McCormick envelope is
\begin{equation}
\mathcal{M}(x, \beta) = \left\{ (x, \beta, \xsf^\betasf) \in [\underbar{x}^{\min},\underbar{x}^{\max}]\times [0,1] \times \R \ : \  
    \begin{aligned}
       \underbar{x}^{\min} \cdot \beta \le  & \, \xsf^\betasf \le \underbar{x}^{\max}\cdot \beta \\
  x - \underbar{x}^{\max}\cdot (1-\beta)  \le &\, \xsf^\betasf \le  x - \underbar{x}^{\min}\cdot (1-\beta) 
    \end{aligned}
    \right\}.
    \label{eq:McCormick}
\end{equation}

More generally, suppose we have $\beta_0, \dots, \beta_m \in \{0,1\}$ such that $\sum_{j=0}^m \beta_j = 1$ hence $\sum_{j=0}^m x\cdot  \beta_j = x$.  Defining $\xsf^{\betasf_j} = x \cdot \beta_j$, we can then model the convex hull of all such solutions as 
\begin{equation}
\mathcal{M}^m(x, \beta) = \left\{ (x, \beta, \xsf^\betasf) \in [\underbar{x}^{\min},\underbar{x}^{\max}]\times [0,1]^{m+1} \times \R^{m+1} \ : \  
    \begin{aligned}
      \sum_{j=0}^m \beta_j=& \, 1, \ \ \  \,\,  
       \sum_{j=0}^m \xsf^{\betasf_j}= \, x\\
       \underbar{x}^{\min} \cdot \beta_j \le  & \, \xsf^{\betasf_j} \le \underbar{x}^{\max}\cdot \beta_j \\
  x - \underbar{x}^{\max}\cdot (1-\beta_j)  \le &\, \xsf^{\betasf_i} \le  x - \underbar{x}^{\min}\cdot (1-\beta_i) 
    \end{aligned}
    \right\}.
    \label{eq:McCormick-general}
\end{equation}
In the special case where $m=1$, the set $\mathcal M^1(x,\beta)$ is a bijective lifting of $\mathcal M(x,\beta)$. 

\subsection{Discretization}
\label{ssec_MILP-approx}
We discuss two methods of linearizing or relaxing  bilinear products: Normalized Multiparametric Descretization Technique (NMDT)~\cite{Castro2015c} and McCormick envelopes.  
The discrete variable $\dot{f}$, contained in the interval $[l,u]$, is converted into binary variables as 
\begin{equation}
\dot{f} = \lambda_0 + \varepsilon \sum_{i=1}^n \sum_{j=0}^{m} (j\, \lambda_i)\cdot \alpha_{ij}, \ \hspace{1cm} \, \sum_{j=0}^m \alpha_{ij} =1 \text{ for all } i=1, \dots, n, \label{eq_disc_idea}
\end{equation}
where $\alpha_{ij} \in \{0,1\}$ for all $i,j$ and $\lambda_i$ are appropriately chosen scalars. Given a desired level of precision $0 < \hat{\varepsilon} \le u-l$, the choices of $\lambda_i$ (and resulting choices for $\varepsilon$, $m$, and $n$) will affect the number of binary variables $\alpha_{ij}$ that are needed to reach the desired precision.

\begin{enumerate}
    \item {Multiparametric Disaggregation Technique} with base $b$ provided $l \geq 0$:
\begin{equation}
 \lambda_0 = 0, \lambda_i = b^{i-1}, m=b-1; \rightarrow  \ 
    \varepsilon = b^{\floor{\log_{b}(\hat{\varepsilon})}}, 
    n = \ceil{\log_{b}(\tfrac{u}{\varepsilon})}. \tag{MDT} \label{eq:MDT}
\end{equation}
     \item  {Normalized Multiparametric Disaggregation Technique, with base $b$}:  
     \begin{equation}
              \lambda_0 = l, \lambda_i = b^{i-1}, m=b-1;
     n = \ceil{\log_b \lrp{\tfrac{u-l
     }{\hat{\varepsilon}}}}, \varepsilon = (u-l) b^{-n}. \tag{NMDT} \label{eq:NMDT}
     \end{equation}
     \item Monolithic Disaggregation:  
     \begin{equation}
         \lambda_0 = l, n=1, \lambda_1 = 1;   
     m = \ceil{\tfrac{u-l}{\hat{\varepsilon}}}, \varepsilon = (u-l)/m. \tag{Mono} \label{eq:mono}
     \end{equation}
\end{enumerate}
\eqref{eq:MDT} discretizes $\dot{f}$  independent of the lower bound $l$, while \eqref{eq:NMDT} and \eqref{eq:mono} use the lower bound and can be interpreted as discretizing $\dot{f}$ after normalizing the bounds $[l,u]$ to $[0,1]$.  Incorporating the lower bound is crucial for reducing the number of binary variables needed (and hence the complexity of the model) to obtain a certain precision $\varepsilon$. 
Furthermore, the exponential approach of \eqref{eq:NMDT} is vastly preferred over the unary approach of \eqref{eq:mono}, allowing only  ($\mathcal{O}(\log(1/\varepsilon))$ as opposed to $\mathcal{O}(1/\varepsilon)$).

In the originating work for \eqref{eq:NMDT} \cite{Castro2015c}, the same number of digits $n$ (and hence, precision $\varepsilon$) is applied to all discretized variables (to maintain the same precision in the \textit{normalized} space), while in the version presented here, $n$ is computed based on the desired precision $\hat{\varepsilon}$ for each individual discretized variable (to maintain similar precision in the \textit{original} space). 
\begin{equation}
\mathcal{D}(x,f, \alpha) = \left\{
 \begin{aligned}(x,f,\dot{f},\Delta f,\\ 
 \alpha, \xsf^\fsf, \xsf^\alphasf)\  
 \end{aligned}:\  
\begin{aligned}
   \, & f = \dot{f} + \Delta f\\
  \,&\dot{f} = \lambda_0 + \varepsilon \sum_{i=1}^n \sum_{j=1}^{m} (j \,\lambda_i)\cdot  \alpha_{ij}&      \\
  \,&\sum_{j=0}^m \alpha_{ij} = 1 &\, i=1, \dots, n\\
    \,&\xsf^\fsf = x \cdot f =\lambda_0 x + x \cdot \Delta f + \varepsilon \sum_{i=1}^n \sum_{j=0}^{m} (j \,\lambda_i)\cdot  \xsf^{\alphasf_{ij}}&       \\
    \,&\xsf^{\alphasf_{ij}} = x \cdot \alpha_{ij} & i=1,\dots, n,\, j = 0, \dots, m\\
   \, & \underbar{x}^{\min} \leq x \leq \underbar{x}^{\max}\\
  \, & \alpha_{ij} \in \{0,1\} & i=1,\dots, n,\, j = 0, \dots, m
 \end{aligned}\right\}.
\end{equation}
Most of the nonlinear products are now linearized and we drop the variables $\dot{f}$ and $f$.  In this way, $f$ is now only recorded implicitly (approximately) through product terms $\xsf^{\alphasf_{ij}}$, $\xsf^\fsf$, and $\xsf^{\Delta \fsf}$. This produces the set
\begin{equation}
\mathcal{\tilde D}(x,f,\alpha) = \left\{ (x, \alpha, \xsf^\fsf, \xsf^{\Delta \fsf}, \xsf^\alphasf) \ :\  
\begin{aligned}
  \,&  \xsf^\fsf = \lambda_0 \cdot x +  \xsf^{\Delta\fsf}+ \varepsilon \sum_{i=1}^n \sum_{j=0}^{m} (j \,\lambda_i)\cdot  \xsf^{\alphasf_{ij}}&      \\
  \, & (x,\alpha_i, \xsf^{\alpha_i}) \in \mathcal M^m(x,\alpha_i)\\
  \, & \alpha_{ij} \in \{0,1\} & i=1,\ldots, n,\, j = 0, \ldots, m
 \end{aligned}\right\},
 \label{eq:tilde-D}
\end{equation}
where $\lambda_0$, $n$, and $\varepsilon$ are computed from a specified $\hat \varepsilon$, $\underbar{x}^{\min}$ and $\underbar{x}^{\max}$.

 The remaining nonlinear product $x \cdot \Delta f$, approximated by $\xsf^{\Delta \fsf}$, will be dealt with in two different ways.

The binarization in \cite{Gupte2016} is similar to \eqref{eq:NMDT} with $b=2$ after projecting out the $\alpha_{i0}$ variables using the equation $\sum_{j=0}^1 \alpha_{ij} = 1$, i.e., $\alpha_{i0} + \alpha_{i1} = 1$.  

\subsubsection{Choosing base $b=2$ reduces variables.}
We suggest the best model uses $b = 2$ since it
uses asymptotically fewer binary variables than any other choice $b > 2$ (as $\hat{\varepsilon} \to 0^+)$.

\begin{proposition} Let $z_{\hat{\varepsilon}}(b) \defeq (b-1) \cdot n = (b-1) \ceil{\log_b\lrp{\dfrac{u-l}{\hat{\varepsilon}}}}$ be the number of binary variables used to represent $\dot{f}$, and let $b_1,b_2$ be two different bases. Then $z'(b_1,b_2) \defeq \ds \lim_{\hat{\varepsilon} \to 0^+} \dfrac{z_{\hat\varepsilon}(b_1)}{z_{\hat\varepsilon}(b_2)} = \dfrac{(b_1-1)\ln(b_2)}{(b_2-1) \ln(b_1)}$.
\end{proposition}
\begin{proof} Asymptotically, as $\hat{\varepsilon} \to 0^+$, we can remove the ceiling function in the definition of $z_{\hat{\varepsilon}}$. Hence
\begin{equation*}
     z'(b_1,b_2) = \ds \lim_{\hat{\varepsilon} \to 0^+} \dfrac
     {(b_1-1) \log_{b_1} \lrp{\dfrac{u-l}{\hat{\varepsilon}}}}
     {(b_2-1) \log_{b_2} \lrp{\dfrac{u-l}{\hat{\varepsilon}}}}
     = \dfrac{(b_1-1)}{(b_2-1)}\log_{b_1}(b_2) = \dfrac{(b_1-1)\ln(b_2)}{(b_2-1) \ln(b_1)}.
\end{equation*}
\end{proof}
Because $z_{\hat \varepsilon}(b)$ is an increasing function of $b$, $z'(b,2)>1$, for all $b>2$.  Using the proposition, we compute $z'(b,2)$ for various values of $b$ in~\ref{tab:my_label}.  In~\cite{Castro2015c}, Castro uses $b=10$.  As the table shows, this binarization will use nearly three times as many binary variables.
\begin{table}[H]
    \centering
        \begin{tabular}{c|ccccccccccccccc}
         $b$ & 3 & 4 & 5 & 6 & 7 & 8 & 9 & 10 \\
         \hhline{-|-----------}
         $z'(b,2)$
             & 1.26186
              & 1.5
              & 1.72271
              & 1.93426
              & 2.13724
              & 2.33333
              & 2.52372
             & 2.70927
    \end{tabular}
    \caption{Asymptotic number of binary variables required for base $b$, as a multiple of the number required for base $2$.}
    \label{tab:my_label}
\end{table}
In summary, we suggest to use NMDT with a base 2 because it will require the fewest number of binary variables to achieve a desired accuracy.

\subsubsection{Implied and redundant equations}

\label{sssec_MILP_var_elim}
Constraint~\eqref{eq_outflow} (and subsequently~\eqref{eq_specOutflow}) couples the variables $\SPL_{k,t},\FEED_{k,t}$, and $\INV_{k,t}$.  For convenience, let us write $x_1 = \SPL_{k,t}$, $x_2 = \FEED_{k,t}$, $x_3 = \INV_{k,t}$, and $f = \SPEC_{k,q,t}$.
That is, 
\begin{align}
    x_1&=x_2+x_3. \label{wcouple}
\end{align}
As a consequence of \eqref{wcouple}, by multiplying through by $f, \Delta f$, and $\alpha_i$,  we obtain three sets of valid equations
\begin{subequations}
\begin{align}
    \xsf^\fsf_1&=\xsf^\fsf_2+\xsf^\fsf_3, \label{vcouple}\\
            \xsf^{\Delta \fsf}_1 &= \xsf^{\Delta \fsf}_2 + \xsf^{\Delta \fsf}_3,\label{ucouple_delta}\\
             \xsf^{\alphasf_i}_1 &= \xsf^{\alphasf_i}_2 + \xsf^{\alphasf_i}_3 & i=1,\dots, n. \label{ucouple}
\end{align}
\end{subequations}
Consider the approximation of $f$ and the resulting product form of the equations
\begin{subequations}
\begin{align}
    f &= l + \Delta f+ \varepsilon \dsum_{i=1}^n 2^{i-1} \cdot \alpha_i\\
    \xsf^\fsf_{\j} &=l \cdot x_{\j} + \xsf^{\Delta \fsf}_\kappa + \varepsilon \dsum_{i=1}^n 2^{i-1}\cdot \xsf^{\alphasf_{i}}_\j && \j =1,2,3 \label{vreps}\\
  \xsf^{\alphasf_{i}}_{\j} &= x_{\j} \cdot \alpha_i && i =1,\dots,n,\  \j = 1,2,3 \label{udef}
\end{align}
\end{subequations}
Of key importance here for model reduction and a stronger formulation are the observations: (1) Bilinear terms can be eliminated using the above equations (2) Applying McCormick relaxations to these bilinear terms before variable elimination produces a tighter formulation.  

First, although other model choices are possible, we project out the variables associated with $\kappa = 1$.  Conveniently, equations \eqref{vreps} and \eqref{udef}, for $\kappa = 1$, are both implied by all other equations.  In fact, 
variables $\xsf^{\fsf}_1$, $\xsf^{\Delta \fsf}_1$ and $\xsf^{\alphasf_i}_1$ by substitution, can be projected out of the model using equations \eqref{wcouple}, \eqref{ucouple_delta}, and \eqref{ucouple}, then the equations \eqref{vreps} and \eqref{udef} become tautologies.   This observation provides convenient reduction choices in the model in terms of both variables and constraints.

Second, applying the projection/linear relaxation in the right order can improve the model.  In particular, let $\mathcal F$ represent the feasible set given by the above equations, let $\mathcal M^{123}(x,\alpha_\kappa)$ be the McCormick relaxation in the space of all variables (similarly, $M^{23}(x,\alpha_\kappa)$ in the space of variables associated with $\kappa = 2,3$), $\proj_{23}$ be the projecting into the space of variables with $\kappa = 2,3$, and $\relax$ denote the removal of bilinear equations \eqref{udef}.

Then we have the strict inclusion
\begin{equation}
\relaxation\left(\proj_{23}\left(\mathcal F \cap  \bigcap_{\kappa = 1,2,3}\mathcal{M}^{123}(x, \alpha_\kappa)\right)\right)\, \subsetneq \, \relaxation\left(\proj_{23}(\mathcal F) \cap \bigcap_{\kappa = 2,3}\mathcal{M}^{23}(x, \alpha_\kappa)\right).
\end{equation}

That is, using the inequalities from the McCormick relaxation from the $\kappa=1$ variables improves the formulation.

\subsection{MILP Models}
We will now outline the two models we propose to consider and then discuss how to integrate them with rolling planning.
\subsubsection{\Center Approximation} \label{sssec_approx_method}
We create a mixed integer linear approximation to the main problem that will only approximately track constraints on specifications.  To check feasibility of the resulting flows $y^{\text{in}}$ and $y^{\text{out}}$, these solutions must be plugged into the main model to recover spec values and ensure the constraints are met.  
The approximation error will be guided by how fine the discrete approximation of the spec variables is.  For the model, a prevision $\hat{\varepsilon}_{q}$ is chosen such that, given an accurate discretization, we should have $|f_{k,q,t} - \dot{f}_{k,q,t}| \leq \hat{\varepsilon}$.  The choice of $\hat{\varepsilon}_{q}$, as discussed earlier, defines the choices of $\epsilon_q, n_q$ and $m_q$ which are implied parameters in the sets $\mathcal{D}(x,f_{k,q,t}, \alpha)$.  Thus, for each spec $q$, there are (potentially) unique $\epsilon, n$ and $m$ that are used for the set $\mathcal{D}(x,f_{k,q,t}, \alpha)$.
We write all products in terms of single variables and we relax \eqref{eq_infspec} to increase the number of feasible solutions to the approximate model.  
We thus model \eqref{eq_specOutflow},\eqref{eq_infspec},\eqref{eq_feedspec_3},\eqref{eq_feedspecrat_3} as
\begin{subequations}
\begin{align}
\SVOLSPL_{k,q,t} - \SVOLFEED_{k,q,t} &= \SVOLINV_{k,q,t} & \,& k \in K, t \in T, \label{eq_specOutflow_disc}\\%\tag{\ref{eq_specOutflow}-appx}\\
	\SVOLSPL_{k,q,t} - \frac{\varepsilon_{q}}{2} \cdot \SPL_{k,q,t}
	&\le \displaystyle \sum_{s \in S_k} \spec_{s,q} \cdot \INB_{s,k,t} + \SVOLINV_{k,q,t-1} & \,&  k \in K, q \in Q, t \in T,\label{eq_infspec_disc_ub}\\% \tag{\ref{eq_infspec}-relax-ub}\\
	\SVOLSPL_{k,q,t} + \frac{\varepsilon_{q}}{2} \cdot \SPL_{k,q,t}
	&\ge \displaystyle \sum_{s \in S_k} \spec_{s,q} \cdot \INB_{s,k,t} + \SVOLINV_{k,q,t-1} & \,&  k \in K, q \in Q, t \in T,\label{eq_infspec_disc_lb}\\%\\\tag{\eqref{eq_infspec}-relax-lb}\\
	\spec^{\low}_{q,t} \cdot \dsum_{k \in K} \FEED_{k,t} &\le  
	\dsum_{k \in K} \SVOLFEED_{k,q,t} \le 
	\spec^{\upp}_{q,t} \cdot \dsum_{k \in K} \FEED_{k,t} & 
	\,& t \in T,~q \in Q, \label{eq_feedspec_3_disc}\\%\tag{\ref{eq_feedspec_3}-appx}\\
	 \specRat^{\low}_{q_1,q_2,t}\cdot \dsum_{k \in K} \SVOLFEED_{k,q_2,t}  &\leq 
	 \dsum_{k \in K} \SVOLFEED_{k,q_1,t} \leq 
	\specRat^{\upp}_{q_1,q_2,t} \cdot\dsum_{k \in K} \SVOLFEED_{k,q_2,t}   
	 & \, & t \in T,~(q_1,q_2) \in Q_{\rat}, \label{eq_feedspecrat_3_disc}%\tag{\eqref{eq_feedspecrat_3}-appx}
\end{align}
\end{subequations}
The remaining constraints are generated using $\mathcal{\tilde D}$ with $\lambda_0 = \spec^{\min}_{k,q}$ and  $\Delta f_{k,q,t} = \frac{\varepsilon_{q}}{2}$ and $b=2$,
\begin{subequations}
\begin{align}
 &(\SPL_{k,q,t}, \alpha_{k,q,t}, \SVOLSPL_{k,q,t},\SVOLSPLdelta_{k,q,t}, \AVOLSPL_{k,q,t} ) \in \mathcal{\tilde D}(\SPL_{k,q,t},f_{k,q,t}, \alpha_{k,q,t} ) &  k \in K, q \in Q, t \in T,\label{eq_tildeD_vmid} \\
 &(\INV_{k,q,t}, \alpha_{k,q,t}, \SVOLINV_{k,q,t},\SVOLINVdelta_{k,q,t}, \AVOLINV_{k,q,t} ) \in \mathcal{\tilde D}(\INV_{k,q,t},f_{k,q,t}, \alpha_{k,q,t} ) &  k \in K, q \in Q, t \in T,\label{eq_tildeD_vend}\\
 &(\FEED_{k,q,t}, \alpha_{k,q,t}, \SVOLFEED_{k,q,t}, \SVOLFEEDdelta_{k,q,t}, \AVOLFEED_{k,q,t} ) \in \mathcal{\tilde D}(\FEED_{k,q,t},f_{k,q,t}, \alpha_{k,q,t})&  k \in K, q \in Q, t \in T. \label{eq_tildeD_yout}
\end{align}
\end{subequations}
Note that each $\alpha_{k,q,t}$ is a vector of $n_{k,q,t}$ binary variables, which may vary in number depending on the accuracy used to approximate $f_{k,q,t}$.

\subsubsection{Redundant equations}
Defining both sets of constraints in \eqref{eq_tildeD_vmid}, \eqref{eq_tildeD_vend} creates some redundant inequalities.  In particular, after omitting the $k,q,t$ indices for convenience, the inequalities 
\begin{subequations}
\begin{align}
& \AVOLSPL_{i,j} \ge \underbar{v}^{\min}_k \cdot \alpha_{i,j}  &   i \in I_{k,q,t}, j \in \{0,1\} \label{eq_AVOLSPLdef_lb}\\
& \AVOLINV_{i,j} \le \underbar{v}^{\max}_k \cdot \alpha_{i,j}   &   i \in I_{k,q,t}, j \in \{0,1\} \label{eq_AVOLINVdef_ub}
\end{align}
\end{subequations}
are redundant due to \eqref{eq_outflow} and the non-negativity of all volume variables.  
\subsubsection{Tightening feasible region to produce feasible solutions\label{sssec_tighten}}
Because the proposed model is an approximation of the original problem, the model might produce a schedule with flows that result in infeasible specs - violating either the spec bounds or spec ratio bounds.  However, in most cases, the magnitude of these violations we observed to be within $\hat{\varepsilon}/2$ for the proposed discretization methods.

Thus, to mitigate this error, we add a buffer of $\hat{\varepsilon}/2$ to each spec requirement, so that the interval $[\spec^{\low}_{t,q}, \spec^{\upp}_{t,q}]$ is replaced with $[\spec^{\low}_{t,q}+\hat{\varepsilon}_q/2, \spec^{\upp}_{t,q}-\hat{\varepsilon}_q/2]$.

As for the spec ratios, we used differentials of the ratio to obtain an upper approximation for the magnitude of the buffer to use: If the spec ratio to control is $\dfrac{\SPEC_{q_1}}{\SPEC_{q_2}}$, where $\SPEC_{q_1}$ and $\SPEC_{q_2}$ are spec qualities to control, the buffer to use is given via the differential:
\begin{equation}
	\Delta\lrp{\dfrac{\SPEC_{q_1}}{\SPEC_{q_2}}} = \dfrac{1}{\SPEC^{\min}_2} \Delta{\SPEC_{q_1}} +  \dfrac{\SPEC^{\max}_1}{\lrp{\SPEC^{\min}_2}^2} \Delta{\SPEC_{q_1}},
\end{equation}
where $\Delta \SPEC_{q_1}$ and $\Delta \SPEC_{q_2}$ are the computed buffers $\hat{\varepsilon}/2$ for $\SPEC_{q_1}$ and $\SPEC_{q_2}$, respectively. Thus, the interval $[\specRat^{\low}_{q_1,q_2,t},\specRat^{\upp}_{q_1,q_2,t}]$ enforced for $\lrp{\frac{\SPEC_{q_1}}{\SPEC_{q_2}}}$ is replaced by $[\specRat^{\low}_{q_1,q_2,t} + \Delta\lrp{\frac{\SPEC_{q_1}}{\SPEC_{q_2}}},\specRat^{\upp}_{q_1,q_2,t} - \Delta\lrp{\frac{\SPEC_{q_1}}{\SPEC_{q_2}}}]$. These buffers proved effective in largely resolving the issue of incorrect feed specs and spec ratios for the problems tested. The buffers could be reduced if they would result in possible ranges for specs or spec ratios that are smaller than the range provided by a single discretization bin (based on initial tank and supply inventories and the parameter $\hat{\varepsilon}_q$ specified by the user). However, for the purposes of this work, we will assume that $\hat{\varepsilon}_q$ is small enough that the resulting ranges for the demand specs are sufficiently large (with width $\ge \hat{\varepsilon}$ after tightening) that feasibility issues do not become problematic for the proposed discretization scheme.

These buffers result in a replacement of the parameters $\spec^{\upp}_{t,q}$, $\spec^{\low}_{t,q}$, $\specRat^{\upp}_{t,q}$, and $\specRat^{\low}_{t,q}$ in Section~\ref{sec:Mathematical-Model} with their `buffered' counterparts, tightened further by the possible feed specs based on initial tank inventory and inbound supply specs.

\subsubsection{McCormick Based Approximation}
\label{sec:NMDT+McCormick}
In a similar way, the final model for the original NMDT method (with McCormick envelopes around the continuous product terms) is given as follows:
\begin{subequations}
\begin{align}
	 \SVOLSPL_{k,q,t} - \SVOLFEED_{k,q,t} &= \SVOLINV_{k,q,t} & \,& k \in K, t \in T \label{eq_specOutflow_disc_NMDT}\\
	\SVOLSPL_{k,q,t} 
	&=\displaystyle \sum_{s \in S_k} \spec_{s,q} \cdot \INB_{s,k,t} + \SVOLINV_{k,q,t-1} & \,&  k \in K, q \in Q, t \in T \tag{\eqref{eq_infspec}-relax}\label{eq_infspec_disc_NMDT}\\
	\spec^{\low}_{q,t} \cdot \dsum_{k \in K} \FEED_{k,t} &\le  
	\dsum_{k \in K} \SVOLFEED_{k,q,t} \le 
	\spec^{\upp}_{q,t} \cdot \dsum_{k \in K} \FEED_{k,t} & 
	\,& q \in Q, t \in T \label{eq_feedspec_3_disc_NMDT}\\
	 \dsum_{k \in K} \SVOLFEED_{k,q_2,t} \cdot \specRat^{\low}_{q_1,q_2,t} &\leq 
	 \dsum_{k \in K} \SVOLFEED_{k,q_1,t} \leq 
	 \dsum_{k \in K} \SVOLFEED_{k,q_2,t} \cdot \specRat^{\upp}_{q_1,q_2,t} 
	 & \, & t \in T,~(q_1,q_2) \in Q_{\rat}  \label{eq_feedspecrat_3_disc_NMDT}
\end{align}
\end{subequations}
Where, for all $k \in K, q \in Q, t \in T$,
\begin{subequations}
\begin{align}
 &(\SPL_{k,q,t},\Delta f_{k,q,t}, \alpha_{k,q,t}, \SVOLSPL_{k,q,t}, \AVOLSPL_{k,q,t} ) \in \mathcal{\tilde D}(\SPL_{k,q,t},f_{k,q,t}, \alpha_{k,q,t} ) &  \KQT \\
 &(\INV_{k,q,t},\Delta f_{k,q,t}, \alpha_{k,q,t}, \SVOLINV_{k,q,t}, \AVOLINV_{k,q,t} ) \in \mathcal{\tilde D}(\INV_{k,q,t},f_{k,q,t}, \alpha_{k,q,t} ) &  \KQT\\
 &(\FEED_{k,q,t},\Delta f_{k,q,t}, \alpha_{k,q,t}, \SVOLFEED_{k,q,t}, \AVOLFEED_{k,q,t} ) \in \mathcal{\tilde D}(\FEED_{k,q,t},f_{k,q,t}, \alpha_{k,q,t})&  \KQT\\
& \Delta \SVOLSPL_{k,q,t} \in [0, \varepsilon \cdot \underbar{v}^{\max}_{k,q,t}] & \KQT\\
& \Delta \SVOLINV_{k,q,t} \in [0, \varepsilon \cdot \underbar{v}^{\max}_{k,q,t}] & \KQT\\
& \Delta \SVOLFEED_{k,q,t} \in [0, \varepsilon \cdot \demand_{k,q,t}] & \KQT
\end{align}
\end{subequations}
Finally, the continuous product terms $\Delta \SVOLSPL_{k,q,t}$, $\Delta \SVOLINV_{k,q,t}$, and $\Delta \SVOLFEED_{k,q,t}$ are modeled via their McCormick envelopes, for all $k\in K, q \in Q, t \in T$,
\begin{subequations}
\begin{align}
\Delta \SPEC_{k,q,t} &\in [0, \varepsilon] & & \KQT\\
(\Delta \SPEC_{k,q,t}/\varepsilon,\SPL_{k,q,t},\Delta \SVOLSPL_{k,q,t}) &\in \mathcal{M}(\Delta \SPEC_{k,q,t}, \SPL_{k,q,t}) & & \KQT\\
(\Delta \SPEC_{k,q,t}/\varepsilon,\INV_{k,q,t},\Delta \SVOLINV_{k,q,t}) &\in \mathcal{M}(\Delta \SPEC_{k,q,t}, \INV_{k,q,t}) & &  \KQT\\
(\Delta \SPEC_{k,q,t}/\varepsilon,\FEED_{k,q,t},\Delta \SVOLFEED_{k,q,t}) &\in \mathcal{M}(\Delta \SPEC_{k,q,t}, \FEED_{k,q,t}) & & \KQT
\end{align}
\end{subequations}

\subsubsection{Error propagation in time horizon}

From a perspective of numerical error propagation, the McCormick-based approach results in an exactly-represented tank-inflow blending process (constraint \eqref{eq_infspec}) for our problem: since only two `spec $\cdot$ volume' terms contain variable specs, the volumes of the specs are coupled correctly. On the other hand, the splitting processes in \eqref{eq_feedspec} are inexact, since the McCormick envelopes around the $\Delta f$ product terms effectively allow the specs leaving a tank to differ slightly from the specs propagated to the next time period. However, due to the tightness of the McCormick envelopes, once a post-blending value $\tilde{f}$ is chosen for a tank, the represented specs will remain in the interval $[\tilde{f}, \tilde{f}+\varepsilon]$ until the next blending operation in the tank, allowing a maximum possible spec error in the interval $[\varepsilon/2, \varepsilon]$ to accrue between blending operations.

On the other hand, the \Center approximation presented here incurs an error of up to $\varepsilon/2$ immediately after blending into the tank, but then accrues no additional numerical error within that tank until the next blending operation. This allows the model to remain closer to feasibility for cases in which each tank does not receive supply from barges very frequently.

\subsection{Rolling Horizon Approach}

As mentioned in Section~\ref{sec:introduction}, we need a blending and scheduling optimizer that can efficiently optimize over a scheduling horizon of at least 3 months and up to 12 months. For the instances of interest, the target time for this optimization is typically preferred to be no more than five minutes (with a ten minute maximum). However, even the linear discrete approximation above is extremely slow even for a scheduling horizon of 120 days, requiring several hours of computation time in Gurobi 8.1, even for extremely course spec discretization (with a total of roughly \textit{six} binary variables per tank per day for only three tanks). 
 As such, we employ a rolling planning approach.

We present three major alternatives for rolling planning. These alternatives are characterized treatment of the `past', `present', and `future', coupled with the choice of time periods (`present') to use for each rolling planning step. The `past' is the portion of the planning horizon before the current rolling horizon window, the `present' is the current horizon window, and the `future' is the portion of the planning horizon after the current rolling horizon window. For the considered schemes, the `future' is split into two partitions: the `near future' and the `far future'. We consider two alternatives for the treatment of the `past' and `future', and two major schemes for the choice of time periods.

We will compare two schedules for rolling planning - \emph{fixed-length} periods and \emph{run-based} periods. In both schemes, periods are chosen to be disjoint since previous performed test with overlapping periods performed slower.

Our fixed-length periods are simply equal length periods of some length $\Delta t$, except the last period may be truncated to not exceed the horizon length $H$.  See Figure~\ref{fig:planning1}.
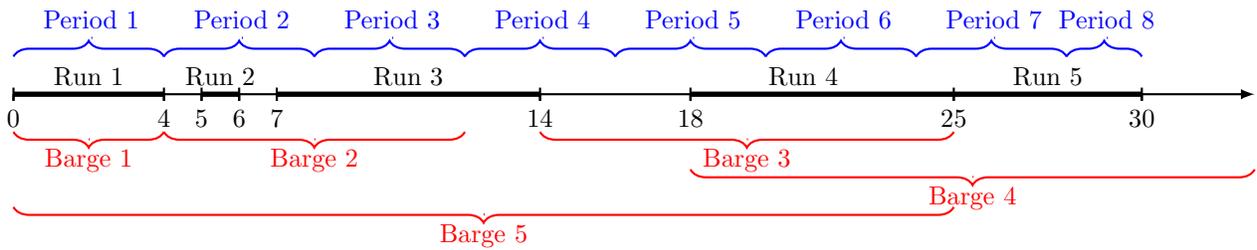
\begin{figure}[H]
    \tikzstyle{barge}=[decoration={brace,amplitude=6pt},decorate, red]
\tikzstyle{periods}=[decoration={brace,amplitude=6pt},decorate, blue]
    \centering
\begin{tikzpicture}[thick,scale=0.5, every node/.style={scale=0.9}]%

\draw[-latex] (0,0) -- (33,0);
\foreach \x in  {0,4,5,6,7,14,18,25,30}
\draw[shift={(\x,0)},color=black] (0pt,5pt) -- (0pt,-5pt);
\foreach \x in {0,4,5,6,7,14,18,25,30}
\draw[shift={(\x,0)},color=black] (0pt,0pt) -- (0pt,-5pt) node[below] 
{$\x$};

\draw[line width= 2] (0,0) -- (4,0);
\draw[color=black] (2,0) -- (2,0) node[above] {Run 1};
\draw[line width= 2] (5,0) -- (6,0);
\draw[color=black] (5.5,0) -- (5.5,0) node[above] {Run 2};
\draw[line width= 2] (7,0) -- (14,0);
\draw[color=black] (10.5,0) -- (10.5,0) node[above] {Run 3};
\draw[line width= 2] (18,0) -- (25,0);
\draw[color=black] (21.5,0) -- (21,0) node[above] {Run 4};
\draw[line width= 2] (25,0) -- (30,0);
\draw[color=black] (27.5,0) -- (27.5,0) node[above] {Run 5};

\foreach \x in {0,4,8,12,16,20,24}
\draw [shift={(\x,0.5)}, periods](0,0.5) -- (4,0.5);

\foreach \x in {1, 2, 3, 4, 5, 6, 7}%
\draw[shift={(\x*4-2,1.5)},color=blue] (2pt,0pt) -- (2pt,0pt) node[above] {Period $\x$};

\draw [shift={(28,0.5)}, periods](0,0.5) -- (2,0.5);
\draw[shift={(29,1.5)},color=blue] (2pt,0pt) -- (2pt,0pt) node[above] {Period $8$};

\draw [barge] (4,-1) -- (0,-1);
\draw[color=red] (2,-1.2) -- (2,-1.2) node[below] {Barge 1};
\draw [barge] (12,-1) -- (4,-1);
\draw[color=red] (8,-1.2) -- (8,-1.2) node[below] {Barge 2};
\draw [barge] (25,-1) -- (14,-1);
\draw[color=red] (19.5,-1.2) -- (19.5,-1.2) node[below] {Barge 3};
\draw [barge] (33,-2) -- (18,-2);
\draw[color=red] (25.5,-2.2) -- (25.5,-2.2) node[below] {Barge 4};
\draw [barge] (25,-3) -- (0,-3);
\draw[color=red] (12.5,-3.2) -- (12.5,-3.2) node[below] {Barge 5};

\end{tikzpicture}
    \caption{This figure shows specified demand runs, the barge time windows, and the periods that we use in the rolling horizon calculation for the fixed-length periods scheme (and considering a horizon $H= 30$ days).}
    \label{fig:planning1}
\end{figure}
Although this period schedule is consistent, it subdivides runs.  This makes modeling consistent volume feed throughout a run \eqref{eq_contfeed} complicated in the rolling planning process and can lead to poor choices early in the rolling planning.    To alleviate this issue, we propose a 'run-based' schedule.

The run-based periods either 
\begin{itemize}
\item Continue in a run to the end of the run or,
\item  Encompass all periods and gaps (days between runs) that are completely contained within a time change of $\Delta t$ from the beginning of period.  
\end{itemize}
This choice makes it so that periods and gaps are never subdivided.   This is depicted in Figure~\ref{fig:planning2}, where a sample partition resulting from this procedure, using $\Delta t=7$, is illustrated.  Formally, our choice of periods is described by Algorithm~\ref{algorithm-run-based}.

\renewcommand{\algorithmicrequire}{\textbf{Input: }}
 \renewcommand{\algorithmicensure}{\textbf{Output: }}
\begin{algorithm}
\algorithmicrequire{Desired approximate length $\Delta t$ for a single rolling planning sub-period}\\
\algorithmicensure{Sub-periods $P_i$ for rolling planning.}
\caption{Rolling Planning with Run Based Sub-periods.}\label{alg:fixed-period}
\begin{algorithmic}%
	\State $t_0\leftarrow 0$, $i \leftarrow 0$ \Comment{Initialize beginning of current rolling planning step and index.}
	\While{$t_i<H$}
		\State $t^{\text{end}} \leftarrow $ first ending time period of a run that is $\geq t_i$
		\State $t^{\text{change}} \leftarrow $ latest ending time period of a run or gap between runs that is $\leq t_i + \Delta t$ 
		\State $t_{i+1} \leftarrow \max(t^{\text{end}},t^{\text{change}}) + 1$
		 \State $P_i \leftarrow [t_i, t_{i+1})$
        \State $i \leftarrow i + 1$
	\EndWhile
\end{algorithmic}
\label{algorithm-run-based}
\end{algorithm}

\begin{figure}[H]
    \tikzstyle{barge}=[decoration={brace,amplitude=6pt},decorate, red]
\tikzstyle{periods}=[decoration={brace,amplitude=6pt},decorate, blue]
    \centering
\begin{tikzpicture}[thick,scale=0.5, every node/.style={scale=0.9}]%

\draw[-latex] (0,0) -- (33,0);
\foreach \x in  {0,4,5,6,7,14,18,25,30}
\draw[shift={(\x,0)},color=black] (0pt,5pt) -- (0pt,-5pt);
\foreach \x in {0,4,5,6,7,14,18,25,30}
\draw[shift={(\x,0)},color=black] (0pt,0pt) -- (0pt,-5pt) node[below] 
{$\x$};

\draw[line width= 2] (0,0) -- (4,0);
\draw[color=black] (2,0) -- (2,0) node[above] {Run 1};
\draw[line width= 2] (5,0) -- (6,0);
\draw[color=black] (5.5,0) -- (5.5,0) node[above] {Run 2};
\draw[line width= 2] (7,0) -- (14,0);
\draw[color=black] (10.5,0) -- (10.5,0) node[above] {Run 3};
\draw[line width= 2] (18,0) -- (25,0);
\draw[color=black] (21.5,0) -- (21,0) node[above] {Run 4};
\draw[line width= 2] (25,0) -- (30,0);
\draw[color=black] (27.5,0) -- (27.5,0) node[above] {Run 5};

\draw [shift={(0,0.5)}, periods](0,0.5) -- (7,0.5);
\draw[shift={(3.5,1.5)},color=blue] (2pt,0pt) -- (2pt,0pt) node[above] {Period $1$};
\draw [shift={(7,0.5)}, periods](0,0.5) -- (7,0.5);
\draw[shift={(10.5,1.5)},color=blue] (2pt,0pt) -- (2pt,0pt) node[above] {Period $2$};
\draw [shift={(14,0.5)}, periods](0,0.5) -- (4,0.5);
\draw[shift={(16,1.5)},color=blue] (2pt,0pt) -- (2pt,0pt) node[above] {Period $3$};
\draw [shift={(18,0.5)}, periods](0,0.5) -- (7,0.5);
\draw[shift={(21.5,1.5)},color=blue] (2pt,0pt) -- (2pt,0pt) node[above] {Period $4$};
\draw [shift={(25,0.5)}, periods](0,0.5) -- (5,0.5);
\draw[shift={(27.5,1.5)},color=blue] (2pt,0pt) -- (2pt,0pt) node[above] {Period $5$};

\draw [barge] (4,-1) -- (0,-1);
\draw[color=red] (2,-1.2) -- (2,-1.2) node[below] {Barge 1};
\draw [barge] (12,-1) -- (4,-1);
\draw[color=red] (8,-1.2) -- (8,-1.2) node[below] {Barge 2};
\draw [barge] (25,-1) -- (14,-1);
\draw[color=red] (19.5,-1.2) -- (19.5,-1.2) node[below] {Barge 3};
\draw [barge] (33,-2) -- (18,-2);
\draw[color=red] (25.5,-2.2) -- (25.5,-2.2) node[below] {Barge 4};
\draw [barge] (25,-3) -- (0,-3);
\draw[color=red] (12.5,-3.2) -- (12.5,-3.2) node[below] {Barge 5};

\end{tikzpicture}
    \caption{This figure shows specified demand runs, the barge time windows, and the periods that we use in the rolling horizon calculation for the `run-based periods' scheme (and considering $H = 30$ days).}
    \label{fig:planning2}
\end{figure}
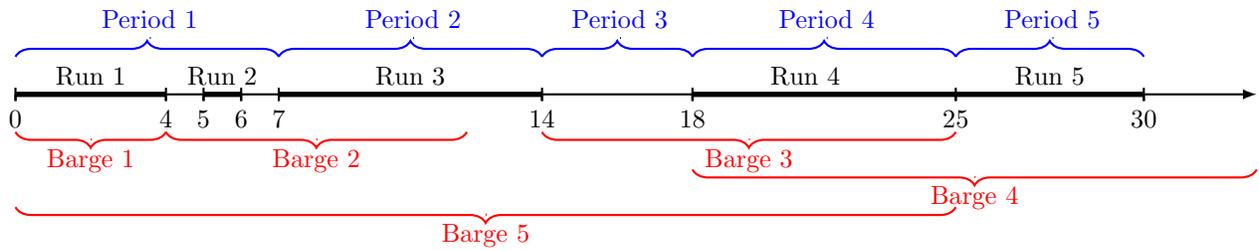

Next, we will discuss the major schemes for the treatment of the past, present, near future, and future segments of the scheduling horizon.  These schemes include a full-horizon scheme and a partial-horizon scheme.

\paragraph{Full-Horizon Scheme}
\newcommand{\nf}{\text{NF}}

Given the current rolling horizon period $P_i = [t_i, t_i+1)$ and a near future horizon $H^{\nf}$, let $t_i^{\nf} = \min(H,t_i + H^{\nf} - 1)$.   In our examples, we use $H^{\nf} = 90$.

\begin{enumerate}
    \item In the past, we freeze the binary decisions made, but allow all continuous variables to continue to vary. This allows the model some additional ability to compensate for poor early decisions, which (as we will see) helps to improve the optimality gap substantially.
    \item In the present, the full model is used.
    \item In the near future, the binary variables related to demand decisions and discretized tank compositions are relaxed to the interval $[0,1]$, while the binary variables related to supply decisions are enforced to be binary. All other constraints are included.
    \item In the far future, all binary variables are relaxed to the interval $[0,1]$. All other constraints are included.
\end{enumerate}

The decision not to relax supply decision binaries in the `near future' is due to the inclusion of the `limited unloading window' constraints for barges. 
These constraints rendered the inclusion of small objective function terms based on supply unloading decisions computationally prohibitive, due to poor branching performance of the solver. Before these were included, small objective function terms had been added to slightly favor earlier unloading times (all else equal), encouraging barges to be unloaded more quickly (and allowing for a model that more consistently avoided small unnecessary supply misses). After the inclusion of these constraints, the solver performance with these small objective terms became extremely poor, requiring at least an order of magnitude more computational time to solve to the same optimality gap in preliminary tests via Gurobi 8.1. When these terms were removed, the model began to exhibit very large extraneous supply misses, due in part to the poor linear relaxation of the `limited unloading window' constraints, and in part to relaxations of the constraints limiting the number of unloads per barge. These poor relaxations caused the rolling planning model to favor unloading during the relaxed-binary periods, a decision which the model could not recover from. However, by including the supply-based binaries in the `near future' (defined to end 90 days after the start of the `present' period), preliminary testing showed that these unnecessary misses became small enough that their cost was considered to be acceptable, while the inclusion of these extra binary variables had only a moderate impact on computational cost (close to a $50\%$ increase).

For all rolling horizon methods, we use 0.5\% as the MIP gap.  For the Full-Horizon scheme, though experimental results we learned that increasing this, on average, does not significantly reduce computation time, but may have a large effect on resulting objective values.  In particular, we want to optimize as much as we can in each step since, due to the way we relax variables, the objective value in each subsequent step can only decrease.  In particular, since we are targeting the best objective value possible, we attempt to make optimal decisions at every step.

\paragraph{Partial-Horizon Scheme}
\begin{enumerate}
    \item In the past, we fix all decisions made, and use them to re-compute accurate values for spec compositions to use for the next rolling planning period. Possible $[l,u]$ for the specs in each tank are re-computed, and the number of remaining unloads (and remaining volumes) for barges that have been unloaded in the past are reduced accordingly.
    \item In the present, the full model is enforced.
    \item In the near future, the binary variables related to demand decisions and discretized tank compositions are relaxed to the interval $[0,1]$, while the binary variables related to supply decisions are enforced to be binary. All other constraints are included.
    \item In the far future, all variables and constraints are omitted. 
\end{enumerate}

\begin{table}[H]
\begin{center}
\begin{tabular}{|c|c|c|c|c|}
\multicolumn{5}{c}{Full Horizon Scheme}\\
\hline
Vars & Past & Present & Near & Far\\
&$\left[0, t_i \right)$ & $P_i$ & $\left[t_{i+1}, t_i^{\nf}\right]$ & $\left(t_i^{\nf},H\right]$\\
\hline
$y, v$ & \tuse & \tuse & \tuse& \tuse\\
$ \ysf, \vsf$ & \tuse & \tuse & \tuse& \tuse\\
$d$  & \tuse & \tuse & \tuse& \tuse\\
$t_s$  & \tuse & \tuse & \tuse& \tuse\\
$\gamma$ & \tfix & \tuse & \tuse& \trelax\\
$\sigma$ & \tfix & \tuse & \trelax& \trelax\\
$\alpha$ & \tfix & \tuse & \trelax& \trelax\\
\hline
\end{tabular} \ \ 
\begin{tabular}{|c|c|c|c|c|}
\multicolumn{5}{c}{Partial Horizon Scheme}\\
\hline
Vars & Past & Present & Near & Far\\
&$\left[0, t_i \right)$ & $P_i$ & $\left[t_{i+1}, t_i^{\nf}\right]$ & $\left(t_i^{\nf},H\right]$\\
\hline
$y, v$ & \tfix & \tuse & \tuse & \tomit\\
$ \ysf, \vsf$ & \tfix  & \tuse & \tuse & \tomit\\
$d$ & \tfix & \tuse & \tuse& \tomit\\
$t_s$ & \tfix & \tuse & \tuse& \tomit\\
$\gamma$ & \tfix & \tuse & \tuse& \tomit\\
$\sigma$ & \tfix & \tuse & \trelax& \tomit\\
$\alpha$ & \tfix & \tuse & \trelax& \tomit\\
\hline
\end{tabular}
\end{center}
\caption{Description of how variables are either fixed, fully active, relaxed to be continuous, or omitted in the two versions of the rolling planning horizon schemes.}
\end{table}

In the partial-horizon scheme, we note that it is possible that only a small part the arrival window, even just a single period, for a single barge overlaps with the `present' and `near future'. To compensate for this, we pro-rate the inventory to be unloaded from such barges: If only $10\%$ of the volume on a barge overlaps with the currently-considered period, then we enforce that only $10\%$ of the volume from the barge is to be unloaded in the current rolling planning step. In this way, as the model rolls forwards, more of the barge's unloading window enters the visible horizon window, allowing for improved decisions regarding the volume on the barge.

\subsection{Post-processing and Error Mitigation}
After solving the approximation model, we check feasibility of the scheduled flows in terms of required spec quality bounds.  To do so, using the flow variables $y^{\text{in}}, y^{\text{out}}$, we plug their computed values into the original model from Section~\ref{sec:Mathematical-Model} to uniquely determine the spec values $f_{k,q,t}$.

After the optimization of the discretized model with rolling planning, we compute specs $f_{k,q,t}$ from 
simulate the scheduling solution to reveal the true specs throughout the scheduling horizon, and report the true tank and feed volumes and compositions throughout the horizon in addition to the basic scheduling plan.

During preliminary computational testing, after simulating the scheduling solution to reveal the true specs throughout the scheduling horizon, we found that, without tightening the spec requirements, the MILP discretization techniques for the model often resulted in solutions for which the actual spec and spec ratios for the demand feeds were not within the specified ranges. However, this was largely remedied after tightening the feasible region as described in Section~\ref{sssec_tighten}

\section{Numerical Results}
\begin{table}
    \centering
    \begin{tabular}{c|c|c|c}
       & Tank 1 &  Tank 2 & Tank 3\\
  Total Volume   &    1179 metric tons &  2948 metric tons & 1225 metric tons\\
  Minimum Volume & 158 metric tons & 272 metric tons & 136 metric tons
    \end{tabular}
    \caption{The capacities of the three tanks we consider.  Minimums are about 10\% of total volumes. Note that barges hold between 1200 - 1400 metric tons.  Hence a barge is comparable to our small tanks, and our large tank can hold slightly more than 2 barges.}
    \label{tab:tanks-and-barges}
\end{table}

\begin{table}[]
    \centering
    \begin{tabular}{c|cccc}
        Barge Type & Type 1 & Type 2 & Type 3 & Type 4 \\
        \hline
      Volume & 1240 & 1360 & 1182 & 1360 \\
      $\misInbPenalty$ & 1000 & 800 & 800 & 1000
    \end{tabular}
    \caption{Example data for barge types, their capacities, and related objective values.}
    \label{table_barge_types}
\end{table}

To evaluate the performance of the proposed methods, we constructed a realistic problem with a 368-day scheduling horizon from real operational data with a 119-day scheduling horizon, extended to the desired horizon in a periodic fashion. The resulting problem had a total of 33 supply arrivals over the scheduling horizon, three available storage tanks, two specs ($S_1$ and $S_2$) to track, and one spec ratio to control. Note that the range of possible spec qualities was relatively narrow for some tanks, resulting in relatively few binary variables used to approximate the specs for each tank.   See \eqref{tab:tanks-and-barges} for example data on the storage tanks.  Other relevant parameters include $\maxUnloadsPerDay = 2$, $\maxUnloadTimeGap = 7$, $\minTankFeedPct_k = 10\%$ for each tank, and $\minDailyUnloadPct = 10\%$ for each supply barge. For the objective coefficients, we globally use $\misPenalty_t = 3000$. $\misInbPenalty_t$ depends on the barge type, and is either \$800 or \$1000 for each type, as specified in Table~\ref{table_barge_types}. Figure \ref{fig:barge_windows_150} showcases the unloading and arrival windows for a sample instance with $H=150$, while Figures~\ref{fig:sln_profiles},\ref{fig:load_unload}, showcase sample solution schedules with $H=40$. Note that the high variability in bounds for the $S1/S2$ ratios is due to the substantially different requirements for product produced in different runs.

Most computational experiments are performed on a Dell desktop with an Intel Xeon 3.2GHz CPU with 8 cores and 16 threads, and 32 GB of RAM, using Gurobi 8.1 in Python 3 on a Windows 10 operations system. The computational experiments in Section~\ref{ssec_grbcomp} are performed on a Dell desktop with an Intel Core i7-7820X 3.6GHz CPU with 8 cores and 16 threads, and 32 GB of RAM, using Gurobi 9.0.1 in Python 3 on a Windows 10 operations system.

\afterpage{
 \begin{landscape}
\begin{figure}
    \centering
    \includegraphics[scale = .45]{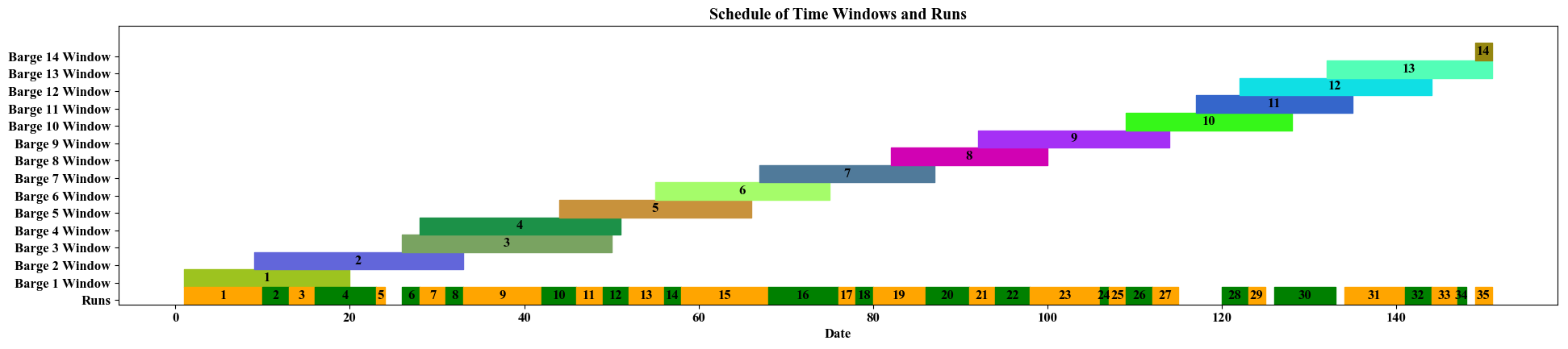}
    \caption{Example of time windows for barge arrivals and dates of runs on a 150 day time horizon.}
    \label{fig:barge_windows_150}
\end{figure}

\begin{figure}
    \centering
    \begin{tabular}{cc}
         \includegraphics[scale = 0.8]{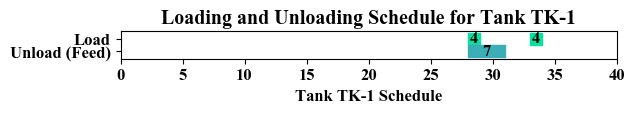}\\
         \includegraphics[scale = 0.8]{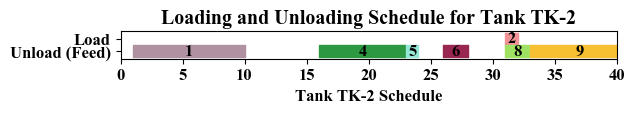}\\
         \includegraphics[scale = 0.8]{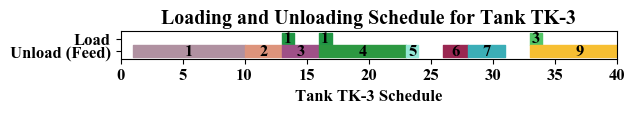}
    \end{tabular}
    \caption{Example of solution for schedules of loading to tanks and  unloading (feed) from tanks to production for a 40 day time horizon.}
    \label{fig:load_unload}
\end{figure}
 \end{landscape}
 }
 
 \begin{figure}
    \centering
    \begin{tabular}{c}
         \includegraphics[scale = .2]{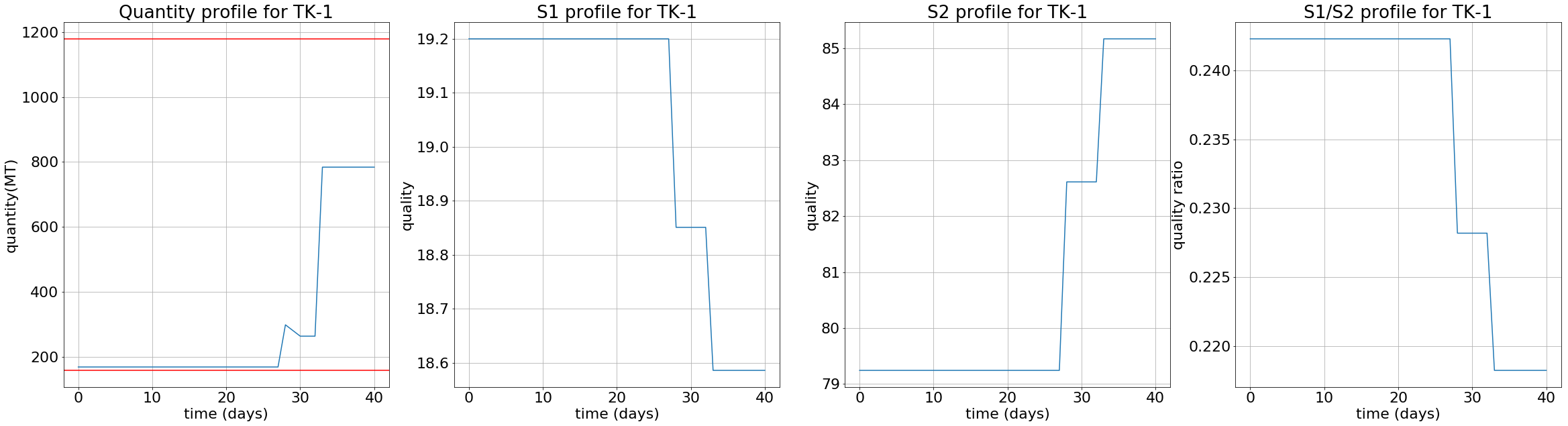}\\
         \includegraphics[scale = .2]{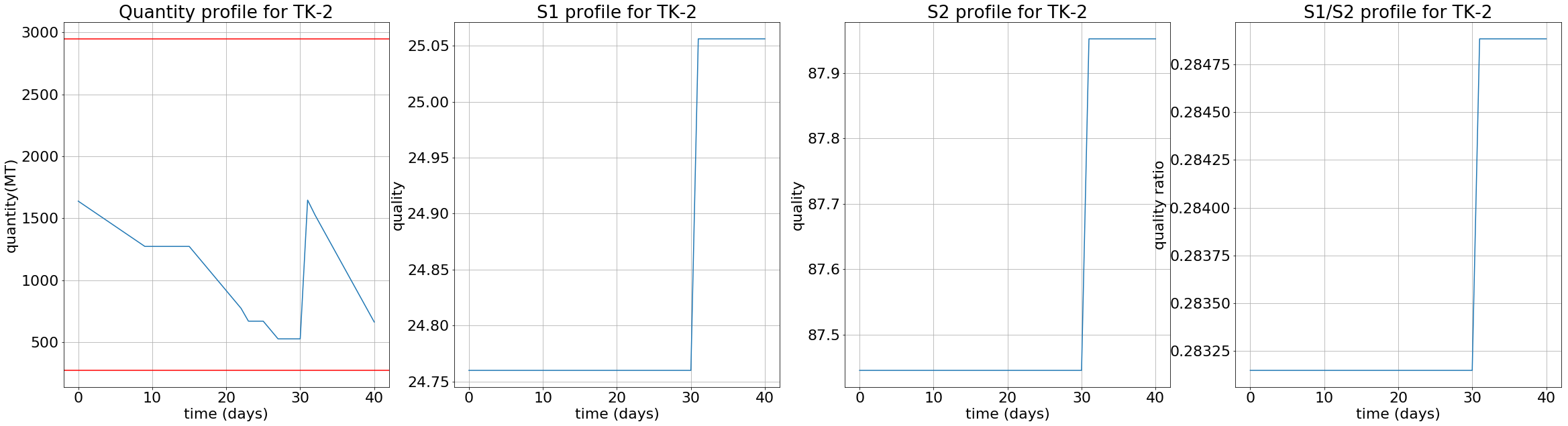}\\
         \includegraphics[scale = .2]{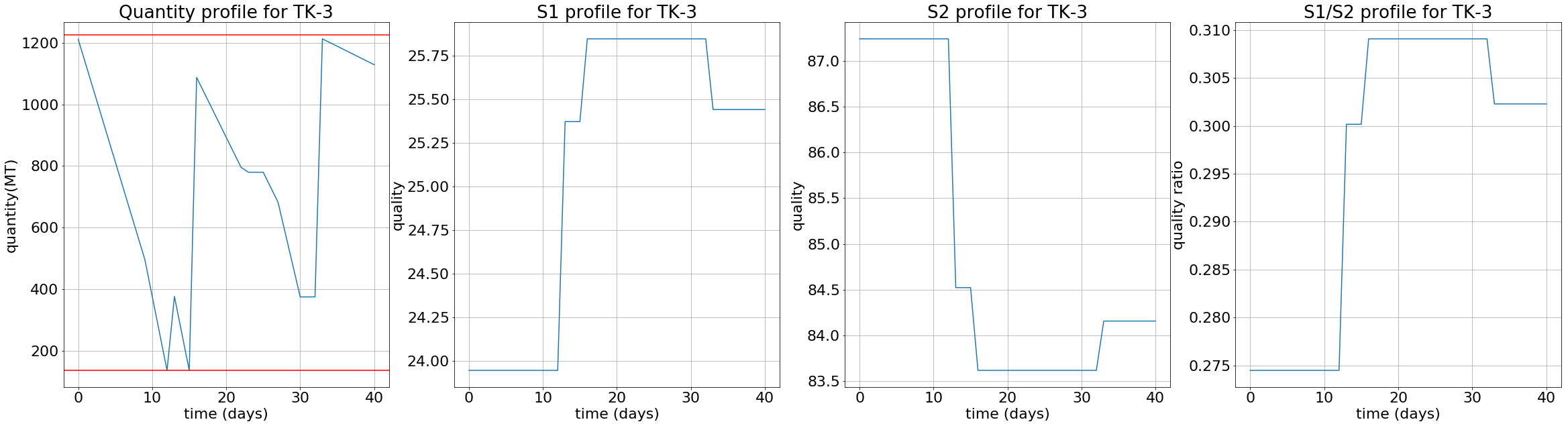}\\
         \includegraphics[scale = .2]{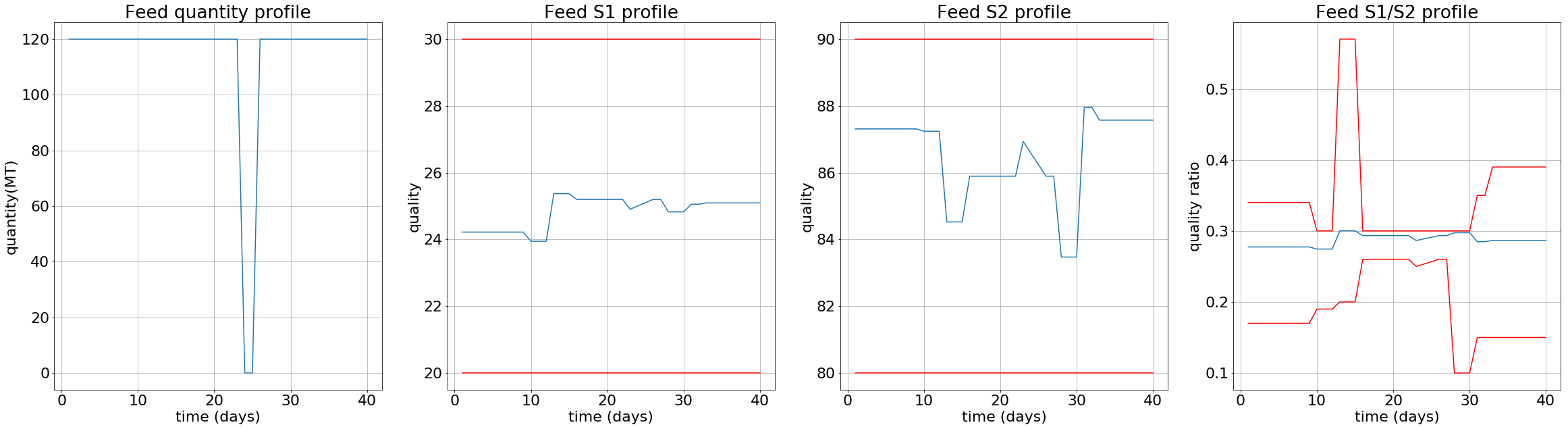}
    \end{tabular}
    \caption{Volume profiles with some bounds describing allowable ranges of spec and spec ratios.}
    \label{fig:sln_profiles}
\end{figure}

\subsection{Comparison with McCormick}
To perform our computational tests, we obtained three years of demand data for the same system, and generated randomized supply data for those three years. We generated ten different such randomizations, and performed each computational test with five different starting dates (separated by half a year), obtaining a total of fifty test problems for each scheduling horizon.

In this section, we compare the effectiveness of two different methods for handling the product terms $v\cdot \Delta f$: McCormick and the proposed \Center method, each using a base-2 binarization of the specs. The tests were repeated using several different values of $H$, and with several values for $\hat{\varepsilon}$ (chosen to be the same for all specs). For $\hat{\varepsilon} = 1$, we used $H \in \{10, 20, 30, 45, 70, 90\}$, and for $\hat{\varepsilon} = 0.25$, we used $H \in \{10, 20, 30, 45, 60\}$ (since McCormick became computationally prohibitive over longer scheduling horizons). A total time limit of $600$ seconds was enforced for all computations in this section. The average-case results are shown in Figure~\ref{fig_benchmarking-NMDT}. In this section, we investigate performance in terms of computation time and percentage of value lost (\%loss) in the objective, with respect to solution values $\MISINB_s$ and $\MIS_t$, computed via
\begin{subequations}
    \begin{align}
        val^{\text{target}} &= \dsum_{s \in S} \misInbPenalty_s \cdot \invInb_s + \sum_{t \in T} \misPenalty_t \cdot \demand_t \label{eq_pctloss_target}\\
        val^{\text{missed}} &= \dsum_{s \in S} \misInbPenalty_s \cdot \MISINB_s + \sum_{t \in T} \misPenalty_t \cdot \MIS_t \label{eq_pctloss_mis}\\
        \%\text{loss} &= \dfrac{val^{\text{missed}}}{val^{\text{target}}} \cdot 100\% \label{eq_pctloss}
    \end{align}
\end{subequations}
In all cases, we see that the proposed \Center method performs much faster than McCormick, with a difference of more than an order of magnitude for longer scheduling horizons. However, this performance improvement came at a cost: as shown in Figure~\ref{fig_benchmarking-NMDT} (b), the average optimality gap for $\hat{\varepsilon} = 1$ yielded by the \Center method is consistently quite a bit worse that that yielded by McCormick. As showcased in Figure~\ref{fig_benchmarking-NMDT-profile}(c), the majority of this difference seems to be caused by significantly worse solutions in roughly a quarter of the test cases. However, this disadvantage disappears at higher fidelity; with $\hat{\varepsilon}=0.25$, the optimality results are very close (well within Gurobi's default optimality gap of $10^{-4}$), with the \Center method taking a slight edge over McCormick. At the same time, the \Center method proved to yield similar feasibility results as McCormick, with each becoming infeasible for no more than one out of the 50 trials in all test cases.

Figure~\ref{fig_benchmarking-NMDT-profile} gives another perspective on the performance advantage of the proposed \Center method, based on the number of instances solved by a given time for the longest scheduling horizons tested: for $\hat{\varepsilon}=1$ and $H=90$, the advantage is overwhelming: almost all of the 50 test instances have converged for the \Center method before any have converged for McCormick. The results for $\varepsilon=0.25$, $H=60$ also show a strong advantage for the proposed method, though this advantage is somewhat less pronounced due to the shorter scheduling horizon.

These results suggest that, for longer scheduling horizons, the new \Center idea fares far better computationally than does the McCormick-style discretization in NMDT for the same level of precision. 

\begin{figure}[H]
	\centering
  \begin{tabular}{cc}
    \includegraphics[scale=0.55]{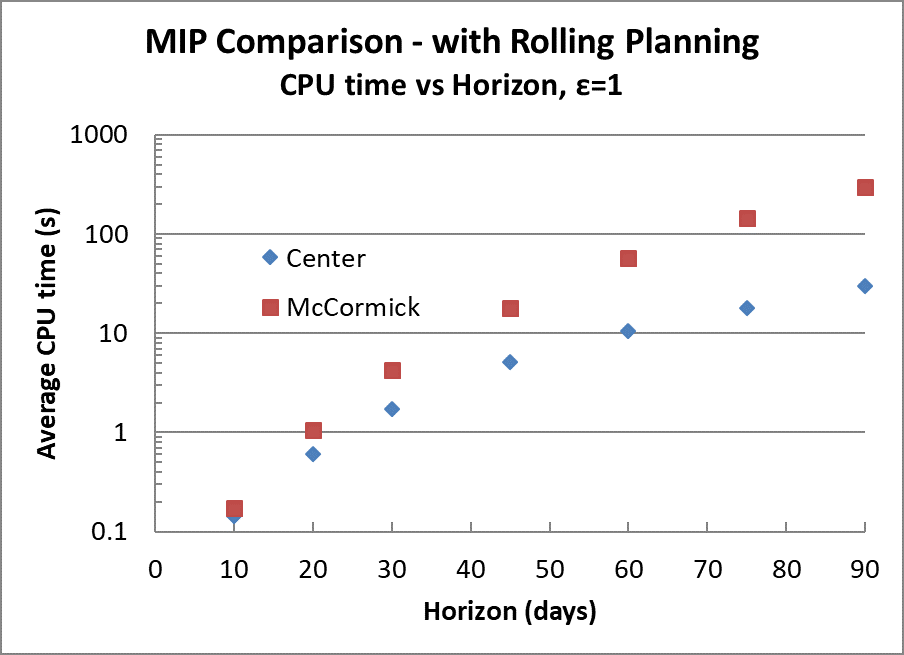} & \includegraphics[scale=0.55]{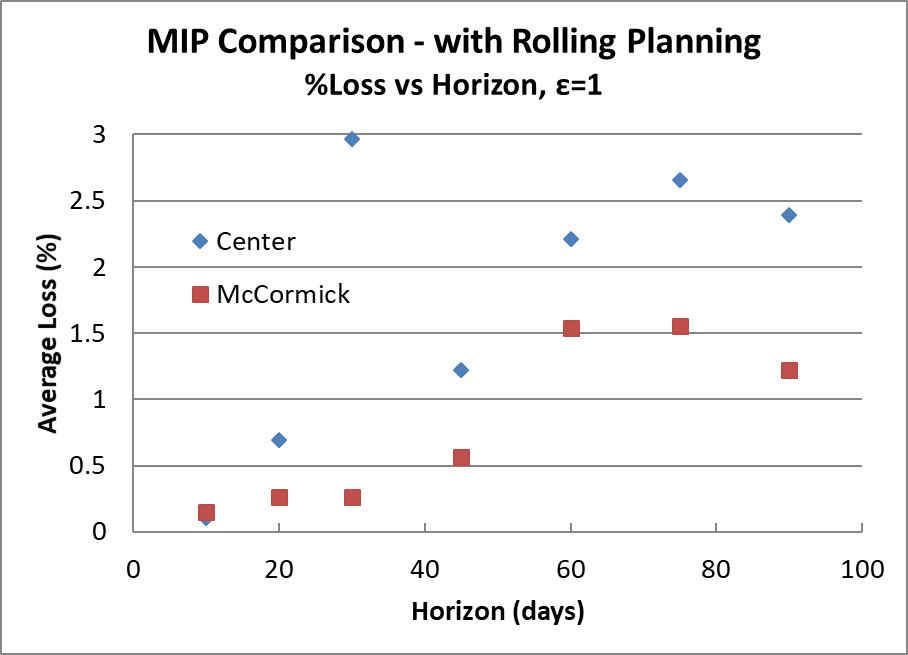}\\
    (a) & (b) \\
    
    \includegraphics[scale=0.55]{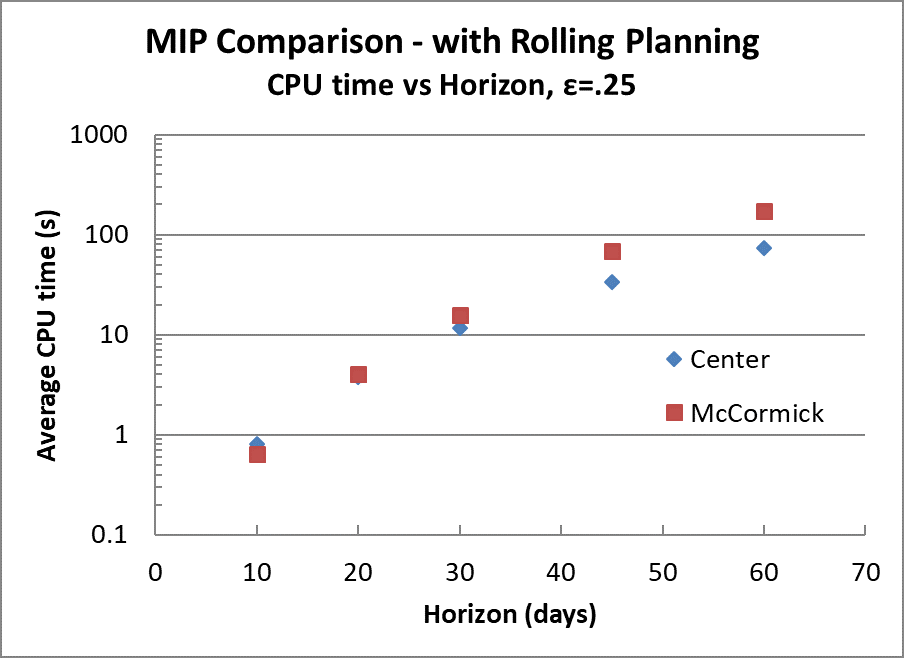} & \includegraphics[scale=0.55]{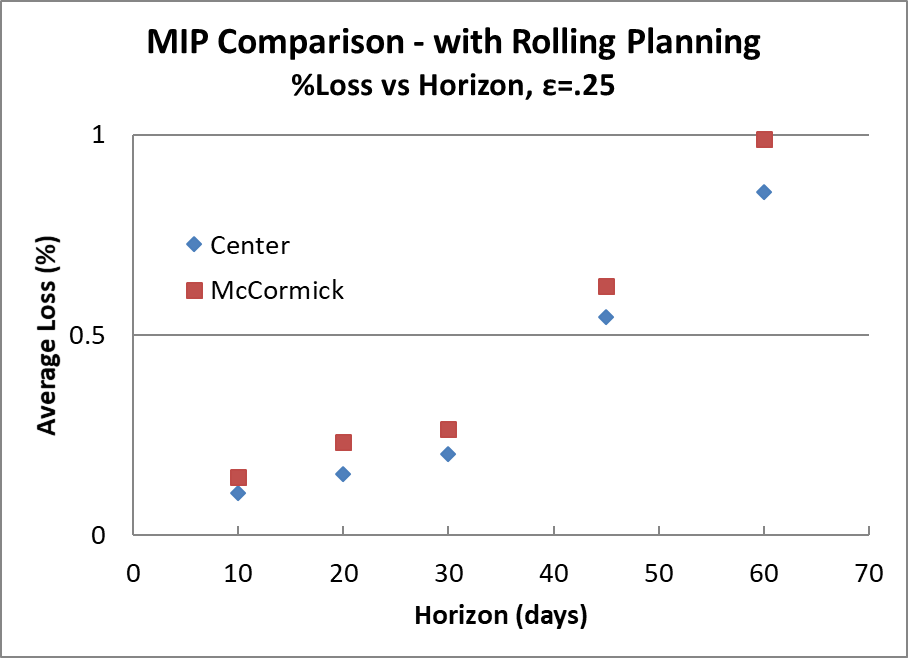}\\
    (c) & (d) 
  \end{tabular}
\caption{Comparison of computational performance of the \Center approximation vs. the McCormick relaxation. Averaged over 50 similar instances with different randomized supply parameters. (a) and (c): Computational time using step sizes of $\hat{\varepsilon} = 1$ and $0.25$. (b) and (d): \% Loss computed as fraction of upper bound on objective value not obtained (see \eqref{eq_pctloss}), where model accuracy is set to $\hat{\varepsilon} = 1$ and $0.25$.}
  \label{fig_benchmarking-NMDT}
\end{figure}

\begin{figure}[H]
	\centering
  \begin{tabular}{cc}
    \includegraphics[scale=0.55]{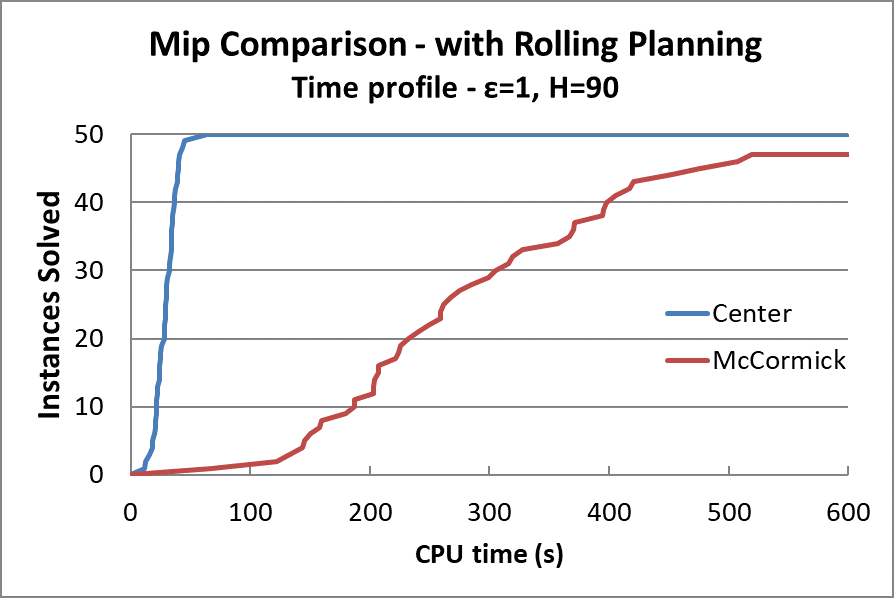} & \includegraphics[scale=0.55]{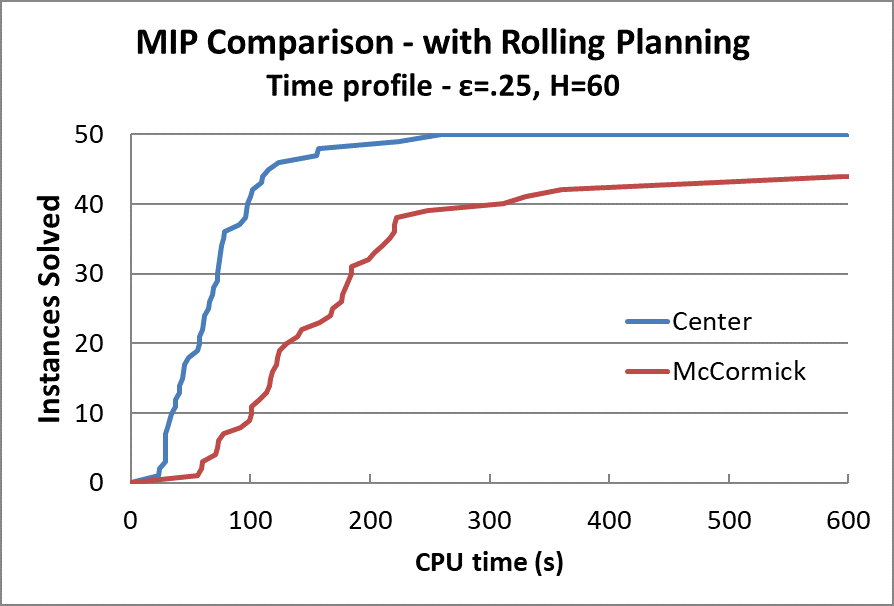}\\
    (a) & (b) \\
    \includegraphics[scale=0.55]{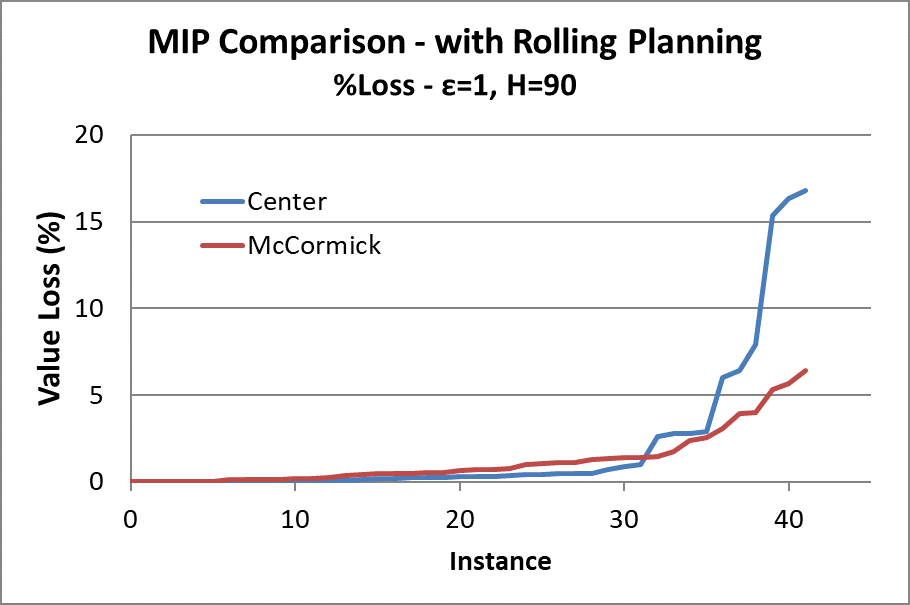} \\
    (c) 
  \end{tabular}
\caption{Time-completion (a,b) and \% value missed (c) profiles for the comparison of computational performance for McCormick vs. the proposed approximate method, using the largest values of $H$ tested for (a,c) $\hat{\varepsilon}=1$, and (b) $\hat{\varepsilon}=0.25$.}
  \label{fig_benchmarking-NMDT-profile}
\end{figure}

Without rolling planning, the difference is far less pronounced: using 30-day and 40-day scheduling horizons with fidelity of $\hat{\varepsilon} = 1$, we find little difference between the performance of the two methods, as showcased in figure Figure~\ref{fig_benchmarking-NMDT-noRP-profile}. As a significant performance difference between the methods had already manifested by such small scheduling horizons in the cases with rolling planning, we find that, for these problem instances, the primary advantage with the Approx method lies in its interaction with rolling planning. 

\begin{figure}[H] 
	\centering
  \begin{tabular}{cc}
    \includegraphics[scale=0.55]{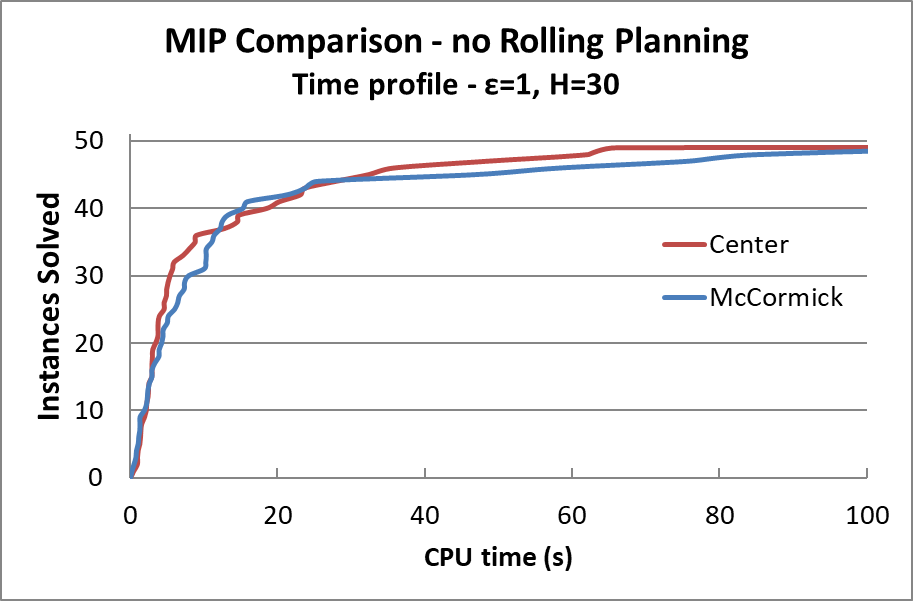} & \includegraphics[scale=0.55]{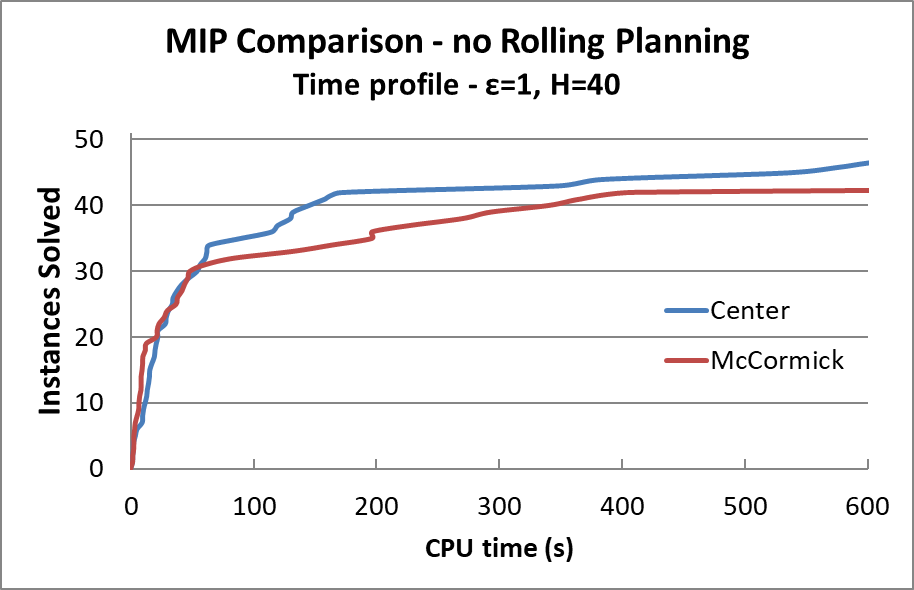}\\
    (a) & (b) \\
  \end{tabular}
\caption{Time-completion profiles for the comparison of computational performance for NMDT vs. the proposed approximate method, using $\varepsilon = 1$ tested for (a) $H=30$ and (b) $H=40$.}
  \label{fig_benchmarking-NMDT-noRP-profile}
\end{figure}

\subsection{Comparison to Nonconvex MIQCP in Gurobi 9.0 }\label{ssec_grbcomp}
In this section, we compare the computational performance of the methods presented, using $\varepsilon=1$ and without rolling planning, to solving the models MINLP-Mix and MINLP-Split defined in  Section~\ref{ssec_prod,ssec_split} with Gurobi 9.0, using scheduling horizons of $H=20$, $H=30$, and $H=40$. The MINLP-Mix model timed out (5-minute limit) for $H=30$ on all instances but one, so we omit it from our comparisons beyond the first plot with $H = 20$.   
The results are shown in Figure~\ref{fig_benchmarking-GRB-noRP-profile}. Note that the MINLP-Mix results are omitted from (b), as many instances also didn't finish for $H=20$. 

The performance of MINLP-Mix is far inferior to that of MINLP-Split: For $H=20$, $\sim$44 instances finished for MINLP-Split within 50s, while only 10 had finished by that point for MINLP-Mix, with only 19 finishing within 5 minutes. For $H=30$, no more than a single instance finish within 5 minutes for MINLP-Mix.

It is interesting to note that, in quite a few instances, Gurobi is able to solve the MINLP-Split model very quickly ($\sim$26 for $H=20$, $\sim$21 for $H=30$, and $\sim$15 for $H=40$. The results show that more instances get stuck with long solution times than with the NMDT-based options, with the slowest $40$-$60\%$ of solutions slower than the NMDT-based methods. However, the solutions it returns are precise in terms of spec feasibility, while the level of solution precision $\varepsilon=1$ for the discrete options is quite coarse. It seems likely that, using to $\varepsilon=0.5$ or $\varepsilon=0.25$, the Gurobi-based solver may overtake the NMDT-based variants in both solution time and quality for these methods. As a result, one might imagine that a rolling planning scheme based on this MINLP-split model, e.g. relaxing the bilinear constraint \eqref{eq_split} for the `future' portion of the horizon, could be quite fruitful. We leave this exploration as future work.

The \%loss results show that, even at such course fidelity, the approximation scheme without rolling planning consistently yields near-optimal results. Indeed, the relative difference between the solutions is consistently within the MIP-gap, 0.5\%, used for the optimization.
\begin{figure}[H]
	\centering
  \begin{tabular}{cc}
  \underline{Computation Time} & \underline{\% Loss} \\
  \\
    \includegraphics[scale=0.21]{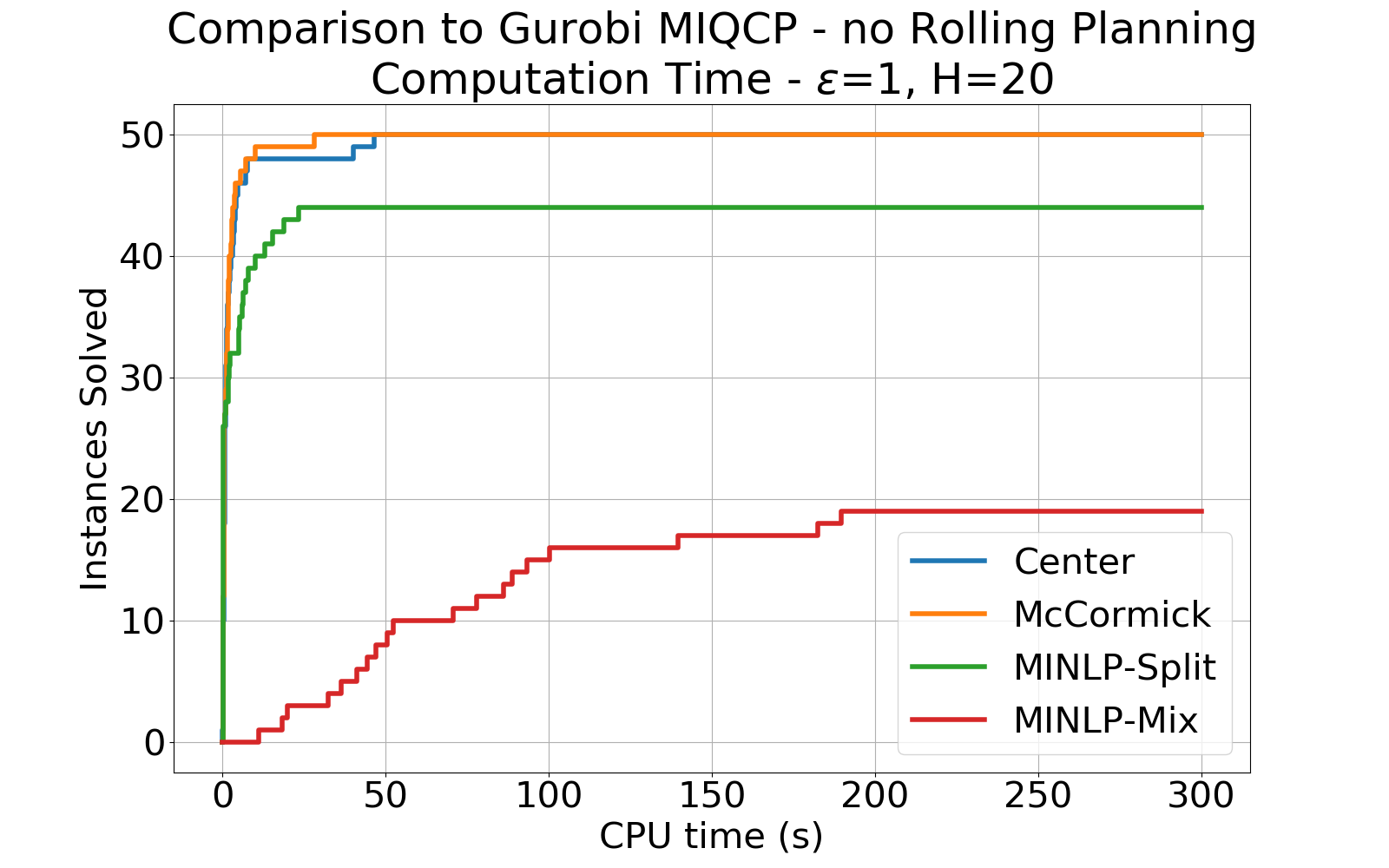} & \includegraphics[scale=0.21]{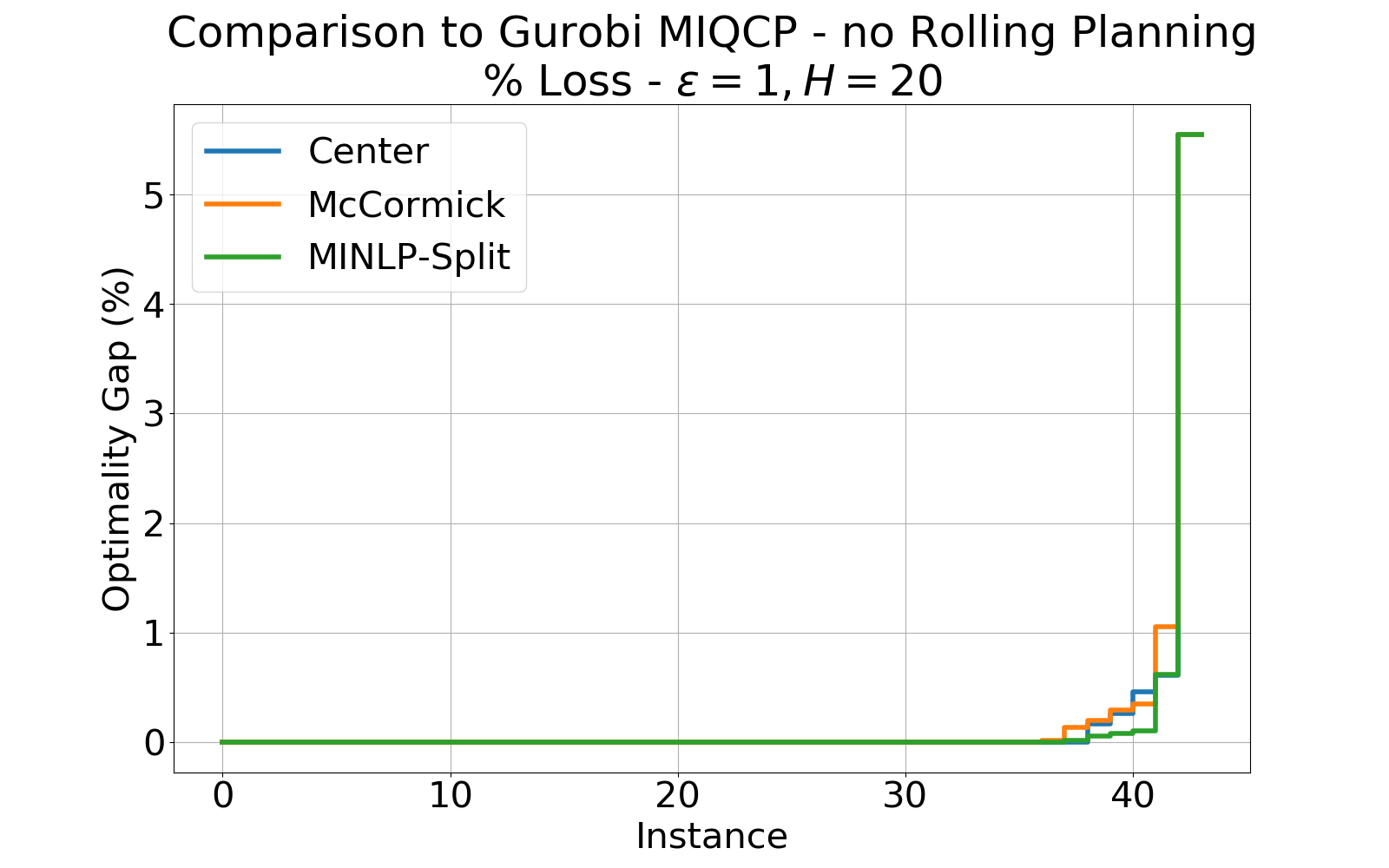}\\
    (a) & (b) \\
    \includegraphics[scale=0.21]{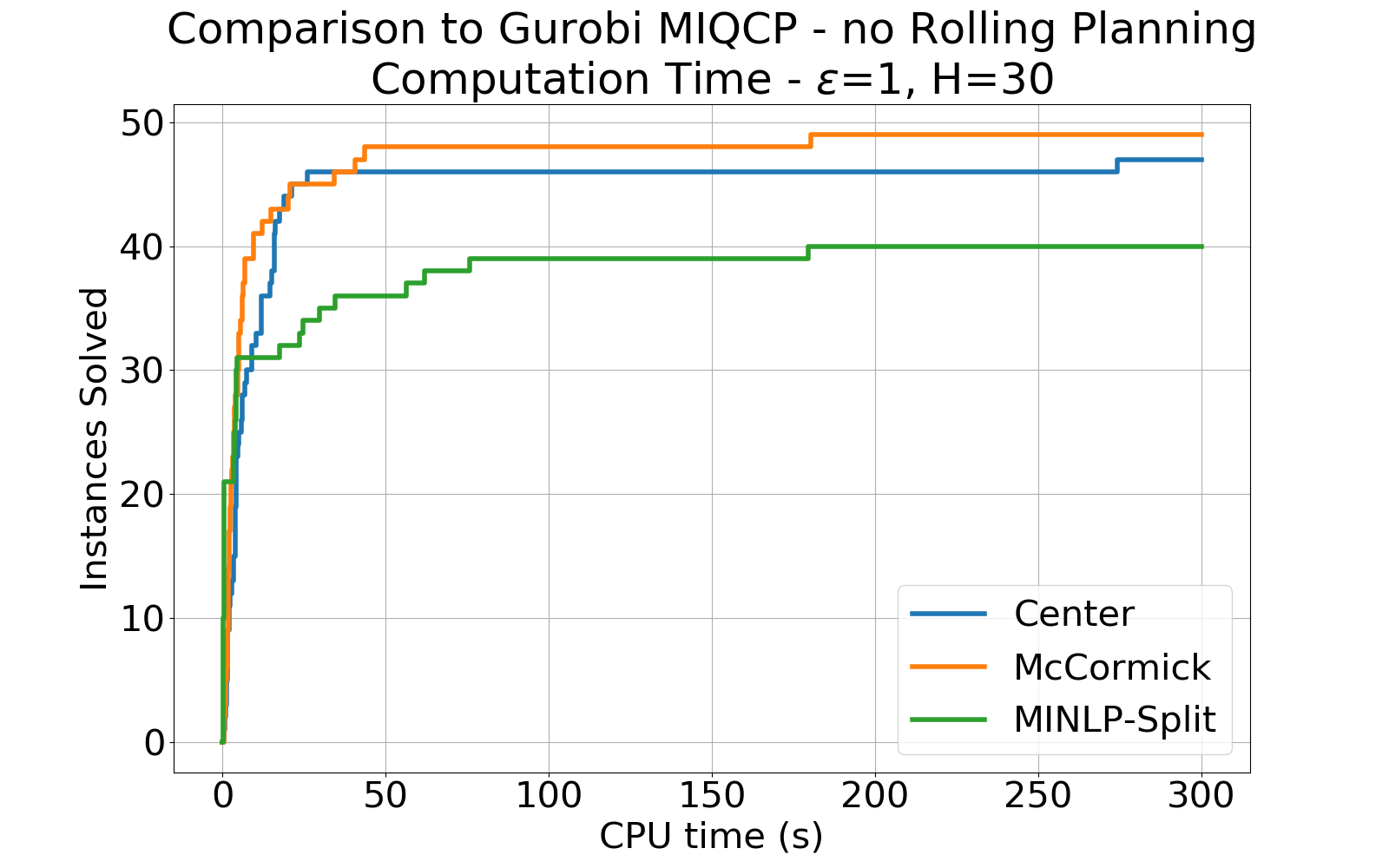} & \includegraphics[scale=0.21]{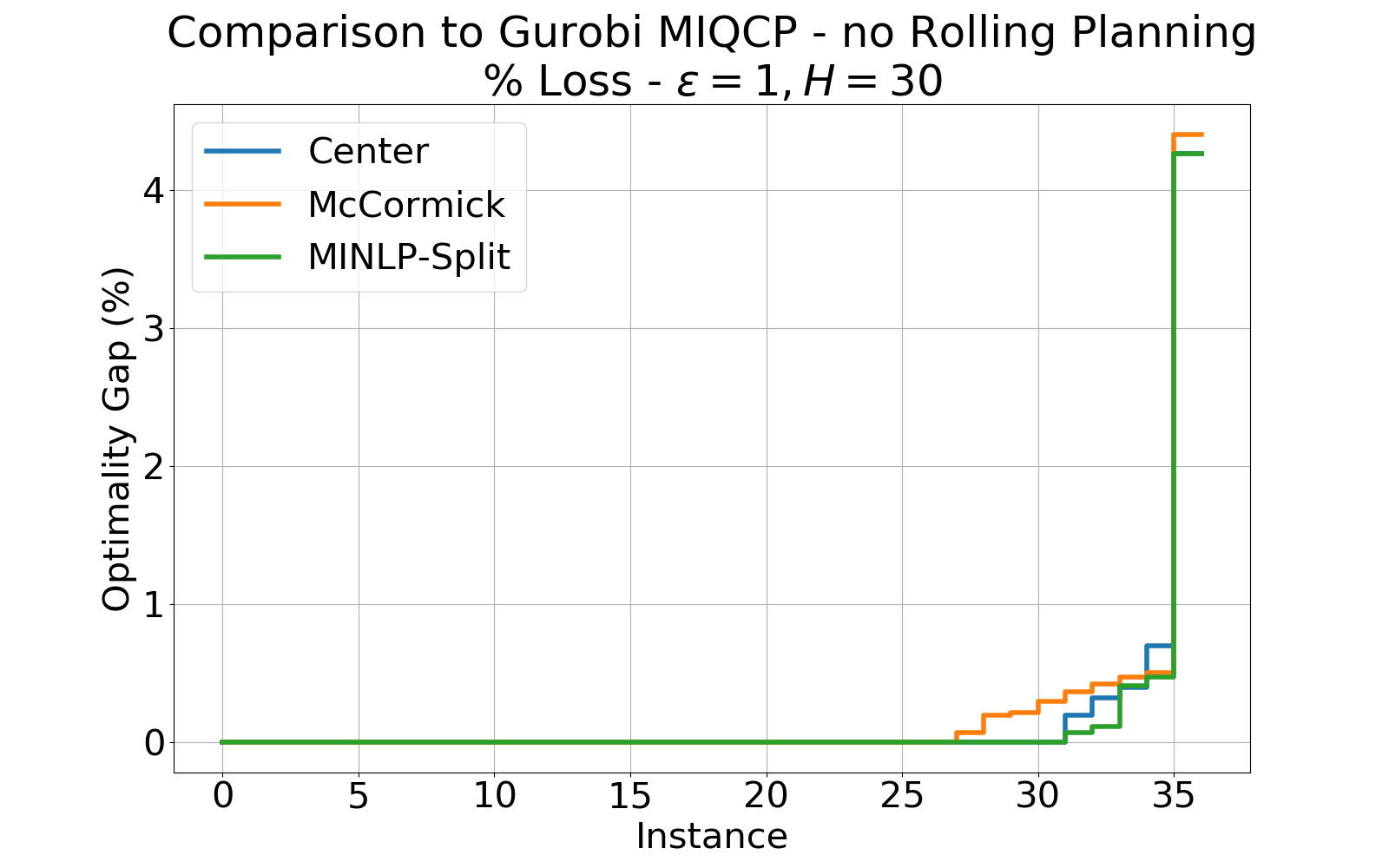}\\
    (c) & (d) \\
    \includegraphics[scale=0.21]{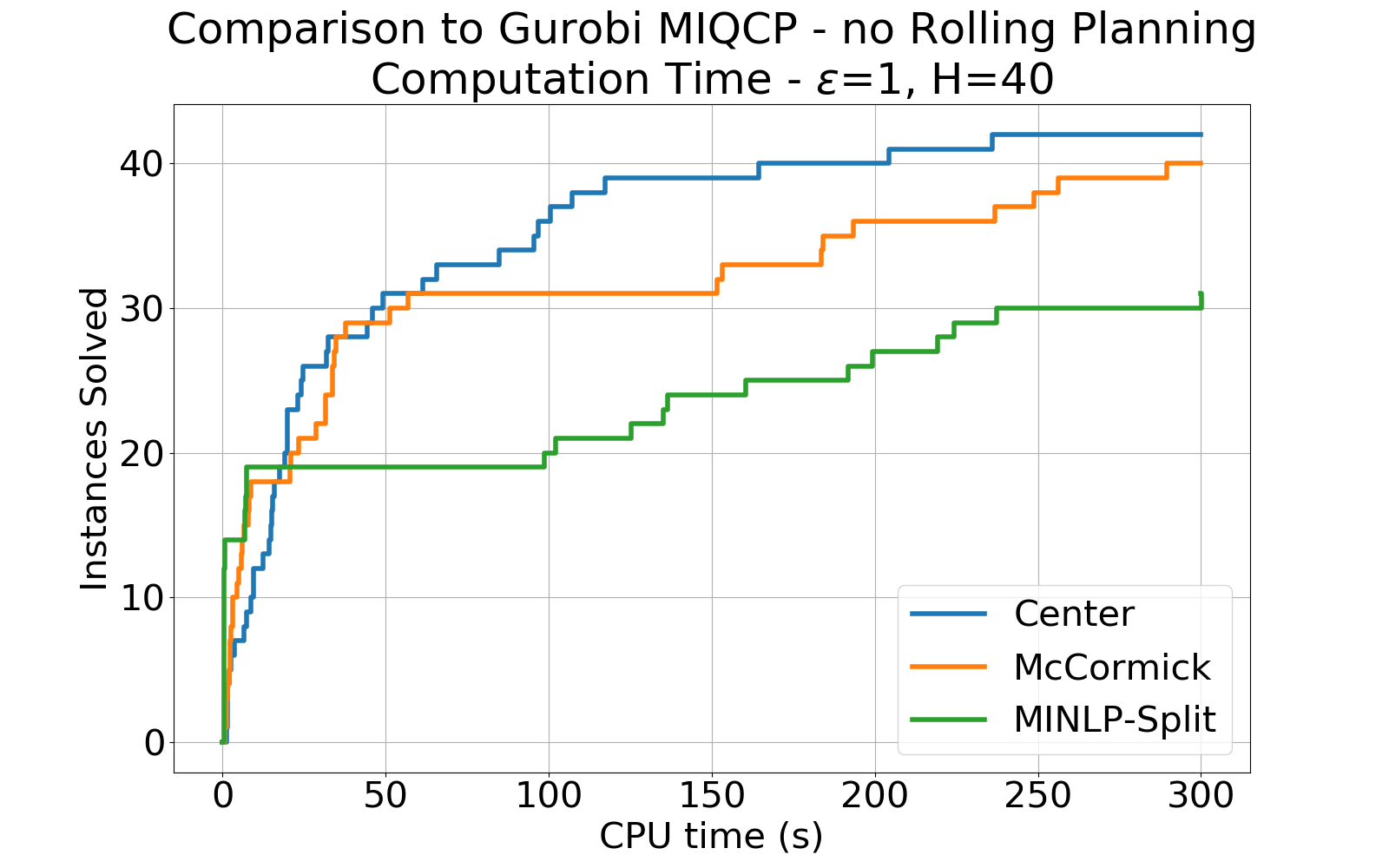} & \includegraphics[scale=0.21]{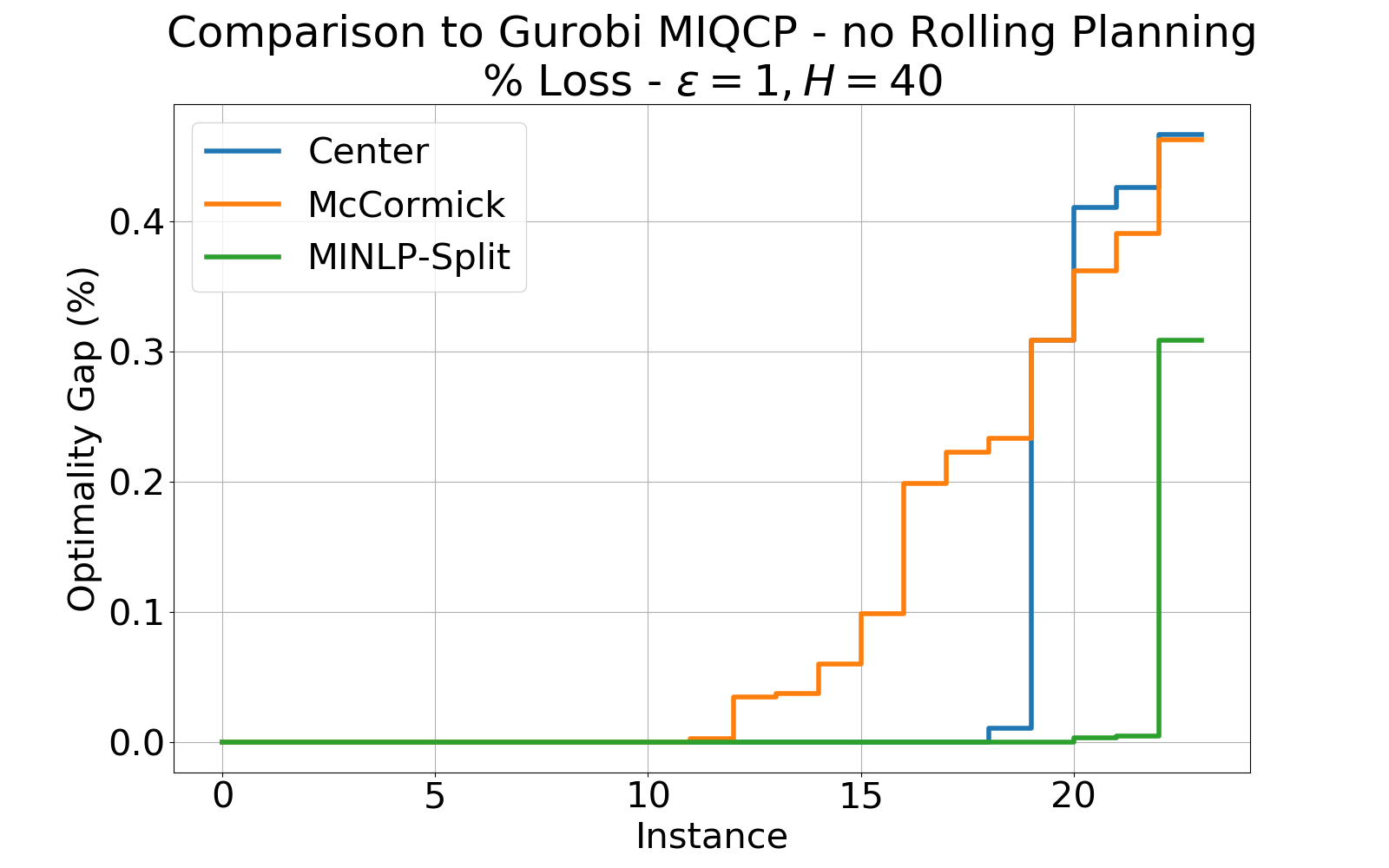}\\
    (e) & (f) \\
  \end{tabular}
\caption{Time-completion (a,c,e) and percent-loss (b,d,f) profiles for the comparison of computational performance for Gurobi 9.0.1 vs. the proposed NMDT-based methods, using $\varepsilon = 1$ tested for (a,b) $H=20$, (c,d) $H=30$, and (e,f) $H=40$.}
  \label{fig_benchmarking-GRB-noRP-profile}
\end{figure}

\subsection{Performance Enhancement}
We explore how to enhance the performance of the method via two simple changes to the \Center model that preserve the mixed-integer feasible solutions. The first of these is to add the coupling constraints described in  Section~\ref{sssec_MILP_var_elim}
\begin{subequations}
    \begin{align}
         & \AVOLSPL_{i,j} = \AVOLINV_{i,j} + \AVOLFEED_{i,j} & i \in I_{k,q,t}, j \in \{0,1\} \label{eq_AVOL_couple}
    \end{align}
\end{subequations}
where $I_{k,q,t}$ is defined as in Section~\ref{sssec_approx_method}. We then remove constraints \eqref{eq_AVOLSPLdef_lb} and \eqref{eq_AVOLINVdef_ub}, since they are now made redundant in the associated LP by \eqref{eq_AVOL_couple} combined with the nonnegativity of $\AVOLSPL_{p,i,1}$, $\AVOLINV_{p,i,1}$, and $\AVOLFEED_{p,i,1}$.\\

The second of these is to relax the MILP approximation of the $\AVOLSPL_{p,i,j}$ terms. This corresponds then removing constraints \eqref{eq_AVOLSPLdef_lb}, \eqref{eq_AVOLINVdef_ub}, and the constraints
\begin{subequations}
\begin{align}
& \AVOLINV_{i,j} \ge \underbar{v}^{\min}_k \cdot \alpha_{i,j}  &   i \in I_{k,q,t}, j \in \{0,1\}, \label{eq_AVOLINVdef_lb}\\
& \AVOLFEED_{i,j} \le \demand_t \cdot \alpha_{i,j}   &   i \in I_{k,q,t}, j \in \{0,1\}, \label{eq_AVOLFEEDdef_ub}
\end{align}
\end{subequations}
which are part of the representation of \eqref{eq_tildeD_vmid} and \eqref{eq_tildeD_yout}. The viability of the removal of \eqref{eq_AVOLFEEDdef_ub} is shown in Section~\ref{sssec_MILP_var_elim}: since any one of the three associated variables $\AVOLSPL_{p,i,1}$, $\AVOLINV_{p,i,1}$, and $\AVOLFEED_{p,i,1}$ can be eliminated completely without compromising the MILP, a valid model is still obtained when relaxing any number of constraints related to the definition of $\AVOLFEED_{p,i,1}$. Constraint $\eqref{eq_AVOLINVdef_lb}$ is always safe to remove, as it is implied by the fact that only one $\AVOLSPL_{i,j}$ variable can nonzero for each $i$, combined with the fact that $\AVOLSPL_{i,j}$ sum to $\SPL_{k,t}$, which is bounded below by $\underbar{v}^{\min}_k$. Note that either  \eqref{eq_AVOLINVdef_ub} or \eqref{eq_AVOLFEEDdef_ub} must be included if the coupling constraints \eqref{eq_AVOL_couple} are not included. In this case, we choose to remove \eqref{eq_AVOLFEEDdef_ub}.

The results, displayed in Figure~\ref{fig_benchmarking-relax-profile}, show that while both modeling changes make a large difference in computational cost, the simple relaxation of the MILP yields a larger improvement in computational time. This improvement is increased further when the coupling constraints are added, yielding a method in which the majority of instances finish within the desired 10 minutes when using a scheduling horizon of a full year with $\hat{\varepsilon} = 1$.

\begin{figure}[H]
	\centering
    \includegraphics[scale=0.55]{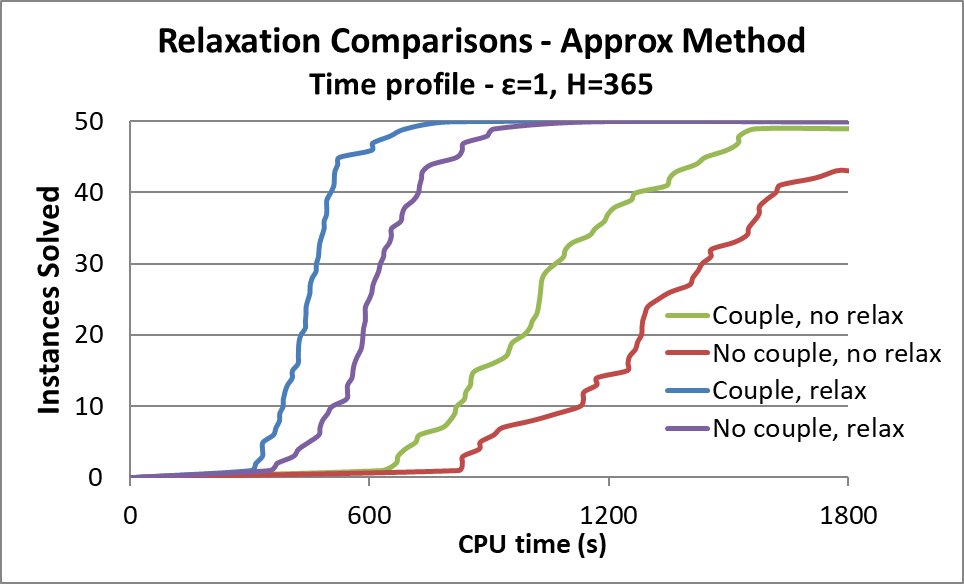}
\caption{Time-completion profiles using different combinations of relaxing (constraints) and adding (coupling constraint), with $H=365$ and a 7-1-1 fixed-step-size rolling planning scheme.}
  \label{fig_benchmarking-relax-profile}
\end{figure}

\subsection{Comparison of Rolling Planning Ideas}
In this section, we compare the computational performance of different rolling planning regimes. We restrict ourselves here to the full-horizon scheme, and compare the performance obtained using two different ideas. For the first, we use fixed-length periods with $H^{\text{int}} = 7$. For the second, we use run-based periods using $\tilde{H}^{\text{int}} = 4$ to generate the periods, with $n^H=\Delta n = 2$ as our stepping parameters. We compare the methods using a time limit of $1800s$ using $\hat{\varepsilon} = 1$ and scheduling horizons of $H \in \{60, 75, 90, 129, 180, 365\}$ (in days).

The results, as shown in Figure~\ref{fig_benchmarking-rolling}, indicate that the regime with a fixed step size yields a more consistent, and often lower, optimality gap on average while incurring a higher computational cost. This outcome is emphasized in Figure~\ref{fig_benchmarking-RP_profiles}: when $H=365$, the 7-1-1 fixed-step scheme results in nearly all test cases finishing within about a 5-10 minute time window, while incurring high optimality gaps in a handful of instances. On the other hand, for the 4-2-2 run-dependent-step scheme, most instances finish within about a 10-20 minute time window, but optimality gap for the worst instances is far more controlled.

\begin{figure}
	\centering
  \begin{tabular}{cc}
    \includegraphics[scale=0.55]{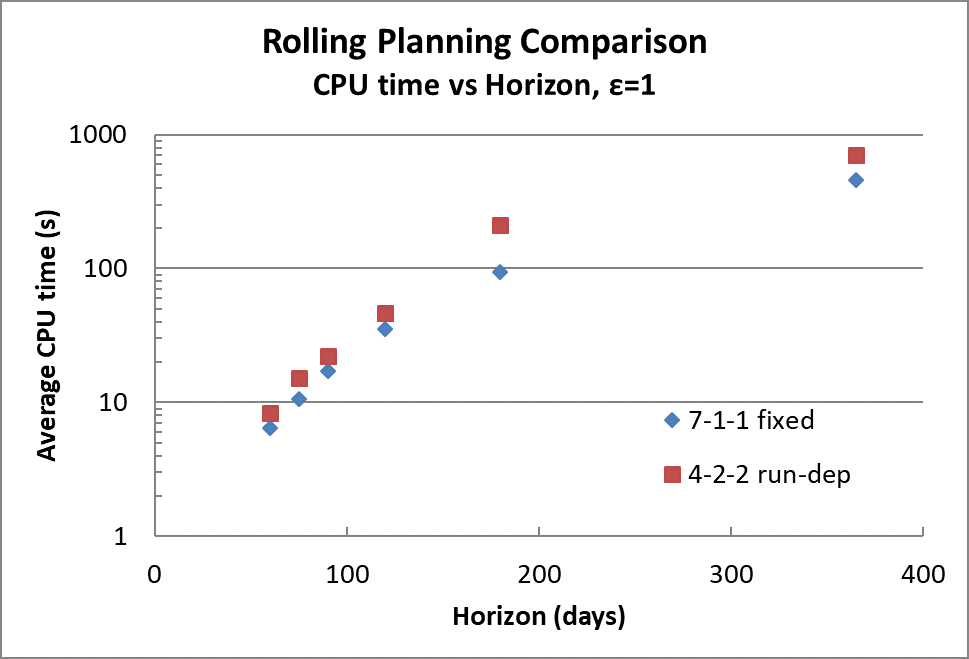} & %
    \includegraphics[scale = 0.55]{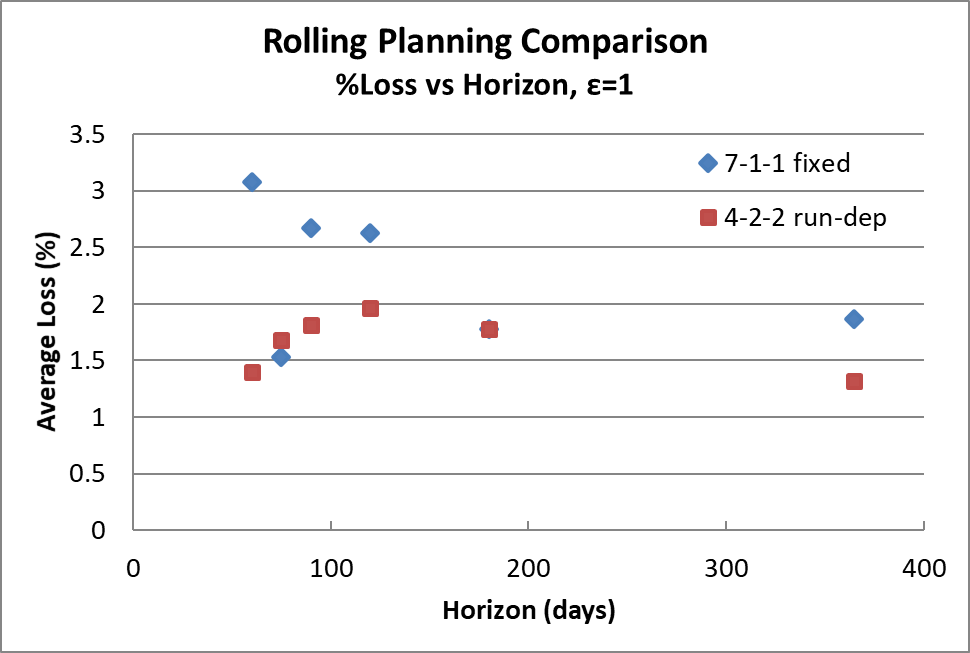}\\
    (a) & (b) \\
  \end{tabular}
\caption{Comparison of fixed-step-length and run-dependent-step rolling planning methods. Averaged over 50 similar instances with different randomized supply parameters, using spec step size $\hat{\varepsilon}=1$. (a): Computational cost, (b) Optimality gap.}
  \label{fig_benchmarking-rolling}
\end{figure}

\begin{figure}
	\centering
  \begin{tabular}{cc}
    \includegraphics[scale=0.55]{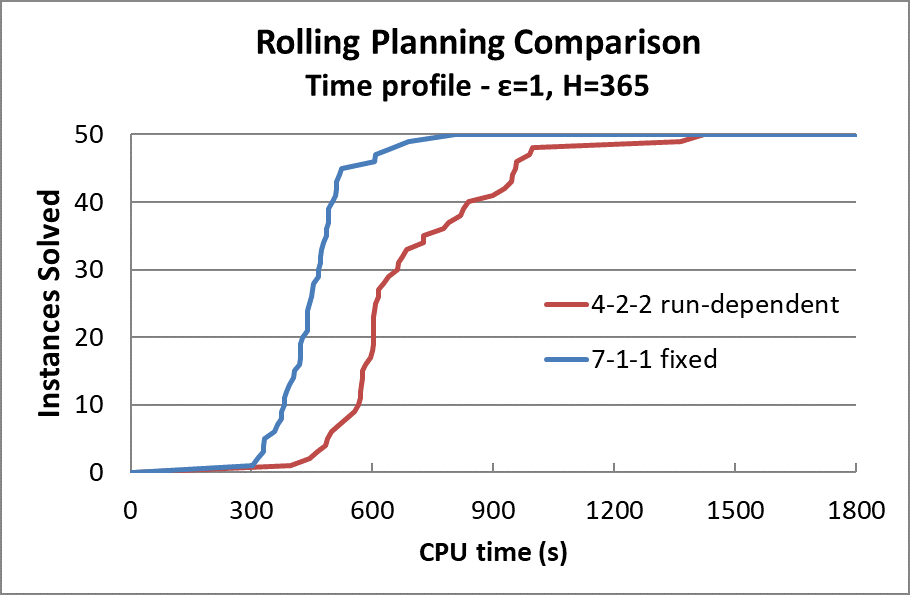} & \includegraphics[scale=0.55]{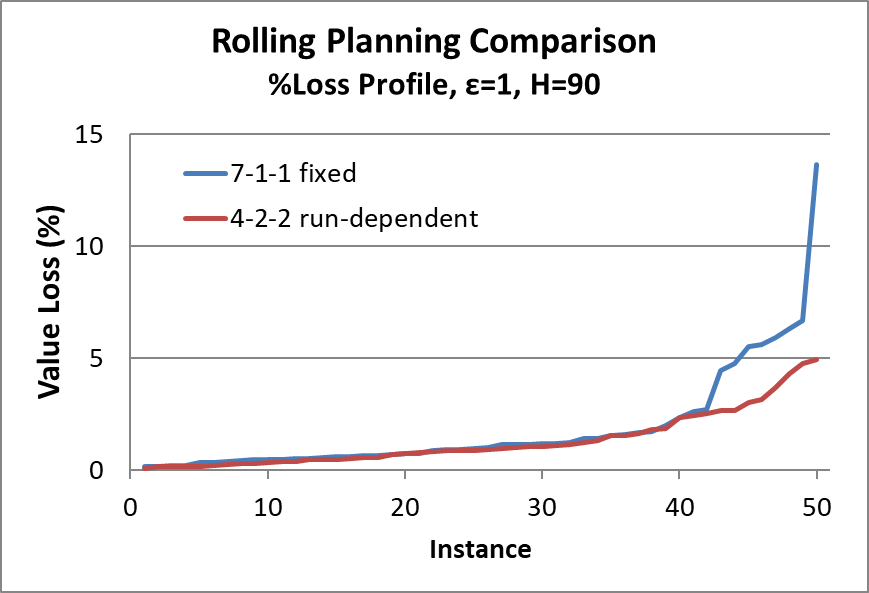}\\
    (a) & (b) \\
  \end{tabular}
\caption{Time-completion profiles of fixed-step-length and run-dependent-step rolling planning methods, using $H = 365$ days and $\hat{\varepsilon}=1$. Run using both the relaxation and coupling things. (a) Computational cost, (b) Optimality gap.}
  \label{fig_benchmarking-RP_profiles}
\end{figure}

\newpage
\section{Conclusion}
In this work, we have developed an approximation method for the tank mixing and scheduling problem that combines rolling horizon planning with two different discretization schemes, yielding reasonably high-quality solutions very quickly over long scheduling horizons, while with high probability ensuring that demand specs remain within the required ranges. Due to this performance, the model could be especially useful for obtaining rough, fast plan feasibility results during the planning of supply acquisitions and product production runs in an industry setting (counting solutions with no (or only small) supply and demand inventory misses as `feasible').\\ %

Comparing the in-house ``Center'' discretization scheme to the original ``McCormick''-based NMDT scheme \cite{Castro2015c}, we find that the in-house method yields much faster performance when combined with a rolling horizon scheme, while sometimes yielding worse solutions. However, without a rolling horizon scheme, the methodologies yield comparable performance, suggesting that the ``Center'' scheme may yield relaxations more compatible with a rolling horizon scheme. Further, when comparing rolling horizon schemes, we find that a scheme using run-based rolling horizon periods yields better solutions than a fixed-duration scheme, at some cost to performance. To extend this work, we plan to explicitly incorporate acquisition planning of barges, and possibly planning of production runs.

\bibliographystyle{plain}
\bibliography{bibliography}

\end{document}